\documentclass[onecolumn]{autart}
\usepackage{}    

\usepackage{tikz} 
\usetikzlibrary{shapes,arrows,chains,backgrounds,fit,decorations.pathreplacing}
\usepackage[europeanresistors,americaninductors]{circuitikz}
\usetikzlibrary{shapes,arrows}
\usepackage{amsmath,bm,times}

\usepackage{mathrsfs}
\usepackage{bbm}    

\usepackage{lineno,hyperref}

\usepackage{graphicx}           
\usepackage{dcolumn}            
\usepackage{bm}                 

\usepackage{graphics}           
\usepackage{epsfig}             
\usepackage{times}              
\usepackage{amsmath}
\usepackage{amssymb}
\usepackage{amsfonts}
\usepackage{fancyhdr}
\usepackage{color}
\usepackage{subfigure}
\usepackage{enumerate}
\usepackage{cite}
\usepackage{wrapfig}
\usepackage{url}
\allowdisplaybreaks[4] 

\newtheorem{theorem}{Theorem} %
\newtheorem{lemma}{Lemma}

\newtheorem{proposition}{Proposition}

\newenvironment{proof}{{\it Proof:\enspace}}{\hfill $\blacksquare$\par}

\newcommand{\divi}{\operatorname{div}}

\begin{document}

%
\begin{frontmatter}

\title{A minimum problem associated with scalar Ginzburg-Landau equation and free boundary} 

\author{Yuwei Hu$^{1}$}\ead{1751839970@my.swjtu.edu.cn},
\author{Jun Zheng$^{1,2}$}\ead{zhengjun2014@aliyun.com},
\author{Leandro S. Tavares$^{3}$}\ead{leandrolstav@gmail.com}

\address{$^{1}${School of Mathematics}, Southwest Jiaotong University, Chengdu, China\\
                $^{2}$Department of Electrical Engineering, Polytechnique
                Montr\'{e}al,  
                Montreal,  Canada\\
                $^{3}$Center for Mathematics, Computing and Cognition,  Federal University of {ABC},  Santo Andr\'{e}, Brazil}  

\begin{abstract}                           
 Let $N>2$, $p\in \left(\frac{2N}{N+2},+\infty\right)$, and $\Omega$ be an open bounded domain in $\mathbb{R}^N$.
      We consider the minimum problem
       $$
       \mathcal{J} (u) := \displaystyle\int_{\Omega }  \left(\frac{1}{p}| \nabla u| ^p+\lambda_1\left(1-(u^+)^2\right)^2+\lambda_2u^+\right)\text{d}x\rightarrow \text{min}
     $$
      over a certain class $\mathcal{K}$, where    $\lambda_1\geq 0$   and $ \lambda_2\in \mathbb{R}$ are constants, and  $u^+:=\max\{u,0\}$.
      The corresponding Euler-Lagrange equation is related to the  Ginzburg-Landau equation and   involves a  subcritical exponent when $\lambda_1>0$.
  For $\lambda_1\geq 0$   and $ \lambda_2\in \mathbb{R}$, we  prove     the existence, non-negativity, and  uniform boundedness of minimizers of $\mathcal{J} (u) $. Then, we show that any minimizer is locally $C^{1,\alpha}$-continuous with some $\alpha\in (0,1)$  and admits the optimal growth $\frac{p}{p-1}$ near the free boundary.  Finally,  under the additional assumption that $\lambda_2>0$, we establish     non-degeneracy for  minimizers near the free boundary  and show that there exists  at least one  minimizer for which  the corresponding free boundary has  finite ($N-1$)-dimensional Hausdorff  measure.
\end{abstract}

\end{frontmatter}
\endNoHyper


 \section{Introduction}

Let $\Omega $ be an open bounded domain in $\mathbb{R} ^N(N>2)$ and $p\in \left(\frac{2N}{N+2},+\infty\right)$. Let {$\lambda_1\geq 0$ and $\lambda_2\in \mathbb{R}$} be  constants.  Given    $g\in W^{1,p}(\Omega)\cap L^\infty(\Omega)$ {with $g\geq 0$ on $\partial\Omega$, we} consider the following minimum problem
\begin{equation}\label{ju}
\mathcal{J} (u){=}\int_{\Omega }  \left(\frac{1}{p}| \nabla u| ^p+\lambda_1\left(1-(u^+)^2\right)^2+\lambda_2u^+\right)\text{d}x\rightarrow \text{min}
\end{equation}
  over the set {$\mathcal{K} :=\left\{u\in W^{1,p}(\Omega );\  u-g\in W_0^{1,p}(\Omega ) \right\}$.

The minimum problem of the type of \eqref{ju} is known as the free boundary problem with  free boundary $\Gamma ^+:=\left(\partial\{x\in \Omega;u(x)> 0\}\right)\cap\Omega $, and has a wide range of applications in various fields. For instance, the minimum problem  of \eqref{ju}
with $\lambda_1=0$ refers to the classical obstacle problem, which  can be used to describe the problems of equilibrium of elastic membranes, fluid filtration in porous {media, and} control of temperature (see \cite{federer1969geometric}), while the minimum problem   \eqref{ju} with $\lambda_1\neq 0$ corresponds to the non-homogeneous scalar Ginzburg-Landau equation,
which arises in the problem of low-temperature superconductivity and has many applications in physics (see \cite{bethuel1994ginzburg,tinkham2004introduction}).  {Taking a mathematical perspective into consideration, it has been pointed out in the work by Ma \cite{li2004obstacle} that the problem represented by equation~\eqref{ju} has a significant connection with the theoretical framework of minimal surface theory.}

In the past few decades, great efforts have been devoted to investigating the existence and regularities of minimizers of $\mathcal{J}(u)$ (or solutions to the corresponding Euler-Lagrange equations) when $\lambda_1=0$ and $\lambda_2>0$,  see for instance \cite{choe1991obstacle,deQueiroz2017,leitao2015regularity,zheng2022free}. {For example in \cite{choe1991obstacle}, the authors considered an obstacle problem for a quasilinear elliptic equation of $p$-Laplacian type. It was obtained that, under certain smoothness conditions on the obstacle function, the solution to the corresponding obstacle problem has interior H\"{o}lder continuous derivatives. In \cite{leitao2015regularity}, a  complete description of the  regularity was provided for a family of heterogeneous, two-phase variational free boundary problems, which are driven  by nonlinear, degenerate elliptic operators. Such study considered a  wide class of problems including, for example, heterogeneous jets and cavities problems of Prandtl-Batchelor type, singular degenerate elliptic equations, and obstacle type systems.  Some of the regularity results obtained in  \cite{leitao2015regularity} were extended to   almost minimizers in \cite{deQueiroz2017} for functionals governed by $A$-Laplacian with  H\"{o}lder continuous coefficients.
In the reference \cite{zheng2022free}, a free boundary  problem  was studied in the Orlicz-Sobolev spaces setting, where the operator  satisfies  structural conditions of Tolksdorf's type and the nonlinearities   may exhibit subcritical growth. Further  geometric properties of the free boundary can be found in  {\cite{edquist2009two,leitao2018free,zheng2018remark}}.}
In particular, it was proved in \cite{caffarelli1998obstacle,lee2003hausdorff}  that the free boundary has finite {($N-1$)}-dimensional Hausdorff measure when $\lambda_1=0$, $\lambda_2>0$, and $p\in [2,+\infty)$. This result was {included} in \cite{challal2012obstacle}  under the framework of Orlicz-Sobolev spaces.
Nevertheless, property of the free boundary in the minimum problem of the type of \eqref{ju} with $\lambda_1\neq 0$  is less studied, except the work presented in \cite{li2004obstacle}, where finite {($N-1$)}-dimensional Hausdorff measure of the free boundary was proved in the non-zero obstacle problem for the Ginzburg-Landau equation when  $\lambda_1>0$, $\lambda_2>0$, and  $p=2$.

In this paper, we  study the minimum problem  \eqref{ju} with  $\lambda_1\geq 0$, $ \lambda_2\in \mathbb{R}$, and  $p$ in a general case.
It is worth noting that, compared with the work in the existing literature, both
the higher order of $u^+$ and nonlinearity of the $p$-Laplacian bring  more intricacies when proving the existence  and regularity of minimizers, as well as Hausdorff measure for the free boundary. This puts forward a constraint on $p$, namely, only the case of $p\in \left(\frac{2N}{N+2},+\infty\right)$ is considered in this paper. Under this condition on $p$, we prove   existence, non-negativity,  and regularity of  minimizers for the functional  $\mathcal{J}(u)$ with $\lambda_1\geq 0$   and $ \lambda_2\in \mathbb{R}$. In addition, we establish  non-degeneracy of   minimizers near the free boundary  and show that there exists {at least one non-negative  minimizer, for which  the corresponding free boundary has  finite {($N-1$)}{-dimensional Hausdorff} measure when   $\lambda _2>0$.}

The rest of this paper is organized as follows.  Notions and some technical lemmas used in this paper are presented in    Section~\ref{sec7}.
 The existence, non-negativity, and uniform  $L^\infty$-boundedness  for the  minimizers of the functional $\mathcal{J}(u)$ are proved in  Section~\ref{sec2}. Local $C^{1,\alpha}$-continuity for the   minimizers and optimal growth  of  minimizers near the free boundary are established in  Section~\ref{sec2'} and Section~\ref{sec3}, respectively.   Under the further assumption that $\lambda_2>0$,    non-degeneracy of  minimizers near the free boundary is proved in Section~\ref{sec3'}, followed which,  a corresponding penalized problem  is studied in   Section~\ref{ssec1}. Finally, for $\lambda_2>0$, we show  in Section~\ref{sec5} that there exists {at least one}    minimizer with finite ($N-1$)-dimensional Hausdorff measure for the corresponding free boundary.
\section{Preliminaries}\label{sec7}
Throughout this paper,  without special statements, we always assume that
 \begin{align*}
 \lambda_1\geq 0,\quad \lambda_2\in \mathbb{R},\quad p\in \left(\frac{2N}{N+2},+\infty\right).
 \end{align*}
Let $\mathbb{N}:=\{0,1,2,3,...\}$ be the set of non-negative integers, $\mathbb{N} ^+:=\{1,2,3,...\}$ be the set of positive integers. Let $p^\star:=\frac{Np}{N-p}$. Denote by $B_r(x)$ a ball in $\Omega $ with centre $x\in\Omega $ and radius $r>0$, and {denote a ball in $\Omega$ by $B_r$ or $B_R$ without special statements on their radius and centres.}  Meanwhile, $\overline{B}_r(x)$ is taken to be the closure of $B_r(x)$.
For a minimizer $u$ of the functional $\mathcal{J} (u)$, let $ \left\{u>0\}:=\{x\in\Omega ;u(x)>0\right\}$ and ${\Gamma ^+=}\left(\partial\{x\in \Omega;u(x)>0\}\right)\cap\Omega $, which is the so-called free boundary.
 For a measurable set $E\subset \mathbb{R} ^N $, $|E|:=\mathcal{L} ^{N}(E)$ denotes the $N$-dimensional Lebesgue measure of $E$, and $\mathcal{H} ^{N}(E)$ denotes the $N$-dimensional Hausdorff measure of $E$. Let $(u)_r:=\frac{1}{|  B_r|}\int_{B_r}  u\text{d}x $ be the average value of the measurable function $u$ on the ball $B_r$ {with radius $r>0$}.

{The result below will be useful in obtaining appropriate $L^{\infty}$ estimates.}
\begin{lemma}\label{jnn}\cite{ladyzhenskaya1968linear}
	Let $\{g_n\}_{n\in \mathbb{N}}$ be a sequence of non-negative numbers satisfying
\begin{align*}
	g_{n+1} \leq CD^{n}g^{1+\zeta}_{n},\ n\in \mathbb{N},
\end{align*}
	where  $C,\zeta>0$ and $ D>1$ are constants not depending on $n$.
	If
\begin{align*}
     g_0 \leq C^{- \frac{1}{\zeta}} D^{-\frac{1}{\zeta^2}},
\end{align*}
	then $g_n \rightarrow 0$ as $n \rightarrow +\infty$.
	\end{lemma}
{In what follows we present a lemma that will be often applied in this manuscript. Such a result can be found in \cite{tav2016thesis}, see also \cite[Lemma A.3]{zheng2022free}.}
    \begin{lemma}\label{phir}
		Let $\tau $ be a non-negative function on an interval $(0,R_0]$ with $R_0 \leq 1$. Suppose that for all $0<r\leq R\leq R_{0}$
		\begin{align*}
      \tau  (r)\leq A\left(\frac{r}{R}\right)^{\alpha}\tau  (R)+BR^{\beta},
		\end{align*}
		where $A>1$ is a constant, $B$, {$\alpha$, and} $\beta$ are constants with {$\beta\in(0,\alpha)$}.
		Fix $\gamma  \in (\beta,\alpha)$ and consider  $\theta \in (0,1)$ with $A\theta^{\alpha} = \theta^{\gamma }$.
		Suppose that there exists a constant $m>0$ such that $\tau  (r)\leq  m  \tau  \left(\theta^{k}R\right)$ for all non-negative integer $k$ and $r \in\left[\theta^{k+1} R,  \theta^{k} R\right]$.
		Then for any $\gamma \in(0,\beta)$, there exists a positive constant $C$ depending on $\tau $, {$\theta$, and} $m$ such that
		\begin{align*}
      \tau  (r) \leq  C\left(\frac{r}{R}\right)^{\gamma }\left(\tau  (R)+BR^{\gamma }\right),
		\end{align*}
		for all $0< r \leq R \leq R_0$, and there is a positive constant $D>0$, which depends on $\tau  $, $C$, $ R_0$, {$\theta$, and} $B$ such that
		\begin{align*}
      \tau  (r) \leq D{r^{\gamma }},
		\end{align*}
		for all ${r\in(0,R_0]}$.
		\end{lemma}
        \begin{lemma}\label{f1l}\cite[Corollary 4.24]{brezis2011functional}
            Let $f\in L^1_{\text{loc}}(\Omega )$ satisfying $\int_{\Omega }  f\varphi \text{d}x=0$ for all $\varphi \in C_0^\infty(\Omega )$.
            Then $f=0$ a.e in $\Omega $.
             \end{lemma}

{The next two results  will be applied in delicate iterations and comparison arguments with appropriate competitors. Moreover, it will be used to obtain interior regularity of the gradient of the minimizers considered.}

    \begin{lemma}\label{nkappa}\cite[Lemma 3.1]{zheng2017regularity}
        Let {$u\in W^{1,p}(\Omega )$,} $B_R\subset \Omega $. If $v$ is a bounded weak solution to
        \begin{align*}
            \divi~\left({|\nabla v|^{p-2}}\nabla v\right)=0\ \text{in}\ B_R,\quad v-u\in W^{1,p}_0(B_R).
        \end{align*}
        Then for any $\mu  \in(0,N)$, there exists a {positive} constant $C$ depending only on $\mu  $, $N$, $p$, {and} $\| v\|_{L^\infty(B_R)}$ such that
        \begin{align*}
            \int_{B_R}  |\nabla u-\nabla v|^p\text{d}x \leq C\int_{B_R}  \left(|\nabla u|^p-|\nabla v|^p\right)\text{d}x +CR^\frac{\mu  }{2}\left(\int_{B_R}  \left(|\nabla u|^p-{|\nabla v|^p}\right)\text{d}x \right)^\frac{1}{2}.
        \end{align*}
     \end{lemma}

     \begin{lemma}\label{nsigma}{\cite[Lemma 4.1]{leitao2015regularity}}
        Let {$u\in W^{1,p}(B_R)$, $B_R\subset \Omega $. If $v\in W^{1,p}(B_R )$} is a weak solution to
        \begin{align*}
        \divi~\left({|\nabla v|^{p-2}}\nabla v\right)=0\ \ \text{in}\ \ B_R.
        \end{align*}
        Then there exist a constant $\delta \in(0,1)$ and a positive constant $C$ depending only on $N$ and $p$ such that
        \begin{align*}
            \int_{B_r}  \left|\nabla u-(\nabla u)_r\right|^p\text{d}x\leq C\left(\frac{r}{R}\right)^{N+\delta  }\int_{B_R}  |\nabla u-(\nabla u)_R|^p\text{d}x+C\int_{B_R} |\nabla u-\nabla v|^p\text{d}x
        \end{align*}
  holds true  for all $r\in(0,R]$.
    \end{lemma}

    \section{Existence, non-negativity, and uniform  boundedness  of   minimizers}\label{sec2}

{In this section, we prove the existence, non-negativity,  and   uniform boundedness  of  minimizers for the functional $\mathcal{J}(u)$.
We have the following result.
      \begin{theorem}\label{t1} The functional of $\mathcal{J} (u)$   admits at least one minimizer over the set $\mathcal{K} $.
      Moreover, every minimizer over the set $\mathcal{K}$ is non-negative {a.e.} in $\Omega$ {and belongs to $ L^\infty(\Omega )$.} Furthermore, there exists a positive constant $C$ depending only on $N$, $p$, $\lambda_1$, $|\lambda_2|$, $\| g\| _{W^{1.p}(\Omega )}$, {$\| g\| _{L^\infty(\Omega )}$, }
       and the diameter of $\Omega $ {such that}
    \begin{align*}
    \| u\| _{L^\infty(\Omega )}+\| u\| _{W^{1.p}(\Omega )}\leq C
    \end{align*}
    {holds true  for all minimizers, still denoted by $u$, of $\mathcal{J} (u)$.}
   \end{theorem}

\begin{proof} The proof consists of four steps, which can be proceeded in a similar way as in \cite[Theorem 2.1]{zheng2022free}.

           {\textbf{Step 1: proof of the existence of a minimizer.}} We {firstly} claim that $I_0:=\displaystyle\inf_{v\in \mathcal{K} } \mathcal{J} (v) >-\infty$.
           Indeed, for $v\in\mathcal{K} $, by the definition of $\mathcal{J} (v)$, we have
       \begin{align}\label{ex1}
           \mathcal{J} (v)\geq \int_{\Omega } \left(\frac{1}{p}|\nabla v|^p+\lambda_1\left(1-(v^+)^2\right)^2-|\lambda_2v|\right) \text{d}x.
       \end{align}
       Since $\lambda_1\geq 0$, it follows that
       \begin{align}\label{ex4}
       \lambda_1\left(1-(v^+)^2\right)^2\geq 0.
       \end{align}

      Note that \begin{align}\label{ex2}
   \int_{\Omega }  |\lambda _2v|\text{d}x
           \leq &  C(N,p,|\lambda _2|,\Omega ) \|v\|_{L^p(\Omega)}  \notag\\
   \leq &C(N,p,|\lambda _2|,\Omega ) \|v-g\|_{L^p(\Omega)}+ C(N,p,|\lambda _2|,\Omega )\|g\|_{L^p(\Omega)} \notag\\
   \leq & C(N,p,|\lambda _2|,\Omega ) \|\nabla v-\nabla g\|_{L^p(\Omega)}+ C(N,p,|\lambda _2|,\Omega )\|g\|_{L^p(\Omega)}\notag\\
   \leq &C(N,p,|\lambda _2|,\Omega ) \|\nabla v \|_{L^p(\Omega)}+ C(N,p,|\lambda _2|,\Omega )\|g\|_{W^{1,p}(\Omega)}\notag\\
   \leq & \varepsilon\int_{\Omega} |\nabla v|^p\text{d}x+ C(\varepsilon,N,p,|\lambda _2|,\Omega )+C(N,p,|\lambda _2|,\Omega )\|g\|_{W^{1,p}(\Omega)},
       \end{align}
   where in the last inequality we used the Young's inequality with {$\varepsilon >0$}.

 By virtue of \eqref{ex1},  \eqref{ex4}, {and} \eqref{ex2}, {and} choosing a suitable $\varepsilon $, we have $I_0>-\infty$.

      Consider a minimizing sequence $\{v_j\}_{j\in\mathbb{N} ^+}$ and $j_0\in\mathbb{N} ^+$ such that $\mathcal{J} (v_j)\leq I_0+1$ for all $j\geq j_0$.
 {By   \eqref{ex2}}, for $\varepsilon >0$ and $j\geq j_0$, we have
       {\begin{align}\label{ex5}
           \int_{\Omega }  \frac{1}{p}|\nabla v_j|^p\text{d}x
           =& \mathcal{J} (v_j)- \int_{\Omega } \left(\lambda_1 \left(1-(v_j^+)^2\right)^2+\lambda_2v_j^+\right) \text{d}x \notag\\
           \leq& \mathcal{J} (v_j)+|\lambda_2|\int_{\Omega }  |v_j|\text{d}x\notag\\
           \leq&I_0+1+C\left(N,p,|\lambda _2|,\Omega\right) \|g\|_{W^{1,p}(\Omega )}  + C(\varepsilon ,N,p,|\lambda _2|,\Omega )+\varepsilon \int_{\Omega }  |\nabla v_j|^p\text{d}x.
       \end{align}}
       A suitable choice of $\varepsilon>0 $ implies that $\int_{\Omega } \frac{1}{p}|\nabla v_j|^p\text{d}x$ is uniformly bounded in $j\in\mathbb{N} ^+$.
       Meanwhile, by the Sobolev embedding inequality we have
       \begin{align*}
          \int_{\Omega }  \frac{1}{p}|v|^p\text{d}x
          \leq& C(p)\int_{\Omega }  \left(|v-g|^p+|g|^p\right)\text{d}x \notag\\
          \leq& C(N,p,\Omega )\int_{\Omega }  |\nabla v-\nabla g|^p\text{d}x+C(p)\int_{\Omega }  |g|^p\text{d}x \notag\\
          \leq& C(N,p,\Omega )\int_{\Omega }  |\nabla v|^p\text{d}x+C(N,p,\Omega )\int_{\Omega } |\nabla g|^p\text{d}x+C(p)\int_{\Omega } |g|^p\text{d}x,
      \end{align*}
       which, along with the uniform boundedness of $\int_{\Omega } \frac{1}{p}|\nabla v_j|^p\text{d}x$, implies
       the uniform boundedness of $\int_{\Omega }  \frac{1}{p}| v_j|^p\text{d}x$. Therefore, $\{v_j\}$ is uniformly bounded in $W^{1,p}(\Omega )$.
       Furthermore, we have $\{v_j-g\}$ is a uniformly bounded sequence in $W^{1,p}_0(\Omega )$.
       Due to the reflexivity of $W^{1,p}_0(\Omega )$, there exists $u\in W^{1,p}(\Omega )$ with $u-g\in W^{1,p}_0(\Omega )$ such that $v_j\rightharpoonup u$ in $W^{1,p}(\Omega )$.
       Then, {there holds} $v_j\to u$ in $L^p(\Omega )$. {Therefore, up to subsequence, we have} $v_j\rightarrow u$ a.e. in $\Omega $.
       Due to the lower semi-continuity of {norms}, we have
       \begin{align}\label{ex6}
           \int_{\Omega }  \frac{1}{p}|\nabla u|^p\text{d}x\leq \liminf_{j\rightarrow \infty} \int_{\Omega }  \frac{1}{p}|\nabla v_j|^p\text{d}x.
       \end{align}
       Since the sequence {$\left\{\left(1-\left(v^+_j\right)^2\right)^2\right\}$} is bounded from below, we infer from the Fatou's Lemma that
      \begin{align}\label{ex7}
           \int_{\Omega }  \left(1-(u^+)^2\right)^2\text{d}x\leq \liminf_{j\rightarrow \infty} \int_{\Omega }   \left(1-(v_j^+)^2\right)^2\text{d}x.
       \end{align}
       Note also that
        \begin{align}\label{ex8}
           \int_{\Omega } u^+\text{d}x\leq \liminf_{j\rightarrow \infty} \int_{\Omega } v_j^+\text{d}x.
       \end{align}
       Therefore, by \eqref{ex6},  {\eqref{ex7}, and} \eqref{ex8}, we deduce that $\mathcal{J} (u)\leq \displaystyle\liminf_{j\rightarrow \infty}\mathcal{J} (v_j)$, which implies that $u$ is a minimizer of $\mathcal{J} (u)$ over the set $\mathcal{K}$.

       {\textbf{Step 2: proof of the non-negativity of  minimizers.}}
   Define $\xi :=\max\{u,0\}=u^+$. By the minimality of $u$ we obtain
       \begin{align*}
      0\leq& \int_{\Omega }  \left(\frac{1}{p}| \nabla \xi| ^p-\frac{1}{p}| \nabla u| ^p+\lambda_1\left(\left(1-(\xi^+)^2\right)^2-\left(1-(u^+)^2\right)^2\right)+\lambda_2\left(\xi^+-u^+\right)\right)\text{d}x\notag\\
%
       = &\int_{\Omega }  \left(\frac{1}{p}| \nabla \xi| ^p-\frac{1}{p}| \nabla u| ^p\right)\text{d}x\\
        =&\int_{{\{u\geq 0\}} }  \left(\frac{1}{p}| \nabla \xi| ^p-\frac{1}{p}| \nabla u| ^p\right)\text{d}x+\int_{{\{u< 0\} }}  \left(\frac{1}{p}| \nabla \xi| ^p-\frac{1}{p}| \nabla u| ^p\right)\text{d}x\\
        \leq& -\int_{{\{u< 0\}} }  \frac{1}{p}| \nabla u| ^p \text{d}x\\
       \leq&{0},
         \end{align*}
     which implies $u\equiv C$ or $u\geq 0$ a.e. in $\Omega $. Since $u=g\geq 0$ {on $\partial \Omega $, we} have $u\geq 0$ {a.e.} in $\Omega$.

       {\textbf{Step 3: proof of the boundedness of   minimizers.}} For each $k\geq k_0:= \max\left\{1, \|g\|_{L^{\infty}(\Omega)}\right\} $, define the function $u_k:\Omega \rightarrow \mathbb{R}$ by
       \begin{equation*}
           u_k=\left\{\begin{aligned}
           & k\cdot\text{sgn}(u)&&\text{if}\ |u|> k,\\
           & u&&\text{if}\ |u|\leq k,
           \end{aligned}\right.
           \end{equation*}
       where $\text{sgn}(u)=1$ if $u\geq 0$ and  $\text{sgn}(u)=-1$ if $u< 0$.
       Let $S_k:=\{|u|>k\}$. For each $k\geq k_0$ we have
       \begin{align*}
           u=u_k \ \text{in} \ {S_k}^c\quad \text{and} \quad u_k = k\cdot \text{sgn}(u) \  \text{in} \ S_k.
       \end{align*}

       Noting $(|u|-k)^+\in W^{1,p}_0(\Omega )$ for $k\geq k_0$ and by the minimality of $u$, {we obtain}
    \begin{align*}
         \int_{S_k}  \frac{1}{p}|\nabla u|^p\text{d}x
        =&\int_{\Omega }  \left(\frac{1}{p}|\nabla u|^p-\frac{1}{p}|\nabla u_k|^p\right)\text{d}x\notag\\
        \leq& \int_{S_k }  \left(\lambda_1 \left(\left(1-(u_k^+)^2\right)^2-\left(1-(u^+)^2\right)^2\right)+\lambda_2(u_k^+-u^+)\right)\text{d}x\notag\\
       =&  {2\lambda_1} \int_{S_k }   \left((u^+)^2-(u^+_k)^2\right)\text{d}x+{\lambda_1 }\int_{S_k }  \left((u^+_k)^4-(u^+)^4\right)\text{d}x+{\lambda_2}\int_{S_k} (u_k^+-u^+)\text{d}x\notag\\
       = & {2\lambda_1} \int_{S_k }   \left((u^+)^2-(u^+_k)^2\right)\text{d}x+{\lambda_1}\int_{S_k\cap \{u\geq 0\}}  \left(k^4-u^4\right)\text{d}x+ {\lambda_2}\int_{S_k\cap \{u\geq 0\}} (k-u)\text{d}x\notag\\
       &+{\lambda_1}\int_{S_k\cap \{u<0\}} \left(((-k)^+)^4-(u^+)^4\right)\text{d}x +{\lambda_2}\int_{S_k\cap \{u<0\}} \left((-k)^+-u^+\right) \text{d}x\notag\\
      \leq &{2\lambda_1} \int_{S_k }   \left((u^+)^2-(u^+_k)^2\right)\text{d}x+ |\lambda_2 |\int_{S_k\cap \{u\geq 0\}} |u-k|\text{d}x\notag\\
      =& 2\lambda_1\int_{S_k}\left(|u|^2-k^2\right)\text{d}x+ |\lambda_2 |\int_{S_k\cap \{u\geq 0\}} ||u|-k|\text{d}x\notag\\
      \leq& 2\lambda_1\int_{S_k}\left(2\left||u|-k\right|^2+2k^2-k^2\right)\text{d}x+ |\lambda_2 |\int_{S_k } ||u|-k|\text{d}x\notag\\
      \leq& 4\lambda_1\left(\int_{S_k}||u|-k|^2\text{d}x+  k^2 |S_k|\right) +\frac{|\lambda_2 |}{2}\left( \int_{S_k}||u|-k|^2\text{d}x+|S_k| \right),
   \end{align*}
               which, along with $k\geq 1$, ensues that
          \begin{align}\label{bou5}
           \int_{S_k}  \frac{1}{p}|\nabla u|^p\text{d}x
           \leq& \left(4\lambda_1+\frac{|\lambda_2 |}{2}\right)\left(\int_{S_k}||u|-k|^2\text{d}x+k^2 |S_k|\right).
          \end{align}

Define $K_n:=\frac{K}{2}\left(1-\frac{1}{2^{n+1}}\right)$ for all $n\in\mathbb{N} $, where $K_n$, $K\geq k_0$. Let $g_n:=\int_{S_{K_n}} \left((|u|-K_n)^+\right)^2\text{d}x$ for all $n\in\mathbb{N}$. We claim that
   \begin{align} \label{iteration-g}
   g_{n+1}\leq CD^ng_n^{1+\zeta },\ \forall n\in\mathbb{N} ,
   \end{align}
   where $C$, $\zeta>0$, and $D>1$ are constants not depending on $n$.

  We prove in two cases:  ${p\in\left(\frac{2N}{N+2},N\right)}$ and $p\in[N,+\infty)$.

   \textbf{Case 1: ${p\in\left(\frac{2N}{N+2},N\right)}$.} Indeed,  when ${p\in\left(\frac{2N}{N+2},N\right)}$, it follows that $2< \frac{Np}{N-p}$, then we have
          \begin{align}\label{bou6}
           g_{n+1}=& \int_{S_{K_{n+1}}}  \left((|u|-K_{n+1})^+\right) ^2\text{d}x\notag\\
           \leq & \left(\int_{S_{K_{n+1}}}  \left((|u|-K_{n+1})^+\right) ^{p^\star}\text{d}x\right)^\frac{2}{p^\star}\left|S_{K_{n+1}}\right|^{1-\frac{2}{p^\star}}\notag\\
           \leq& \left(\int_{\Omega }  \left((|u|-K_{n+1})^+\right) ^{p^\star}\text{d}x\right)^\frac{2}{p^\star}\left|S_{K_{n+1}}\right|^{1-\frac{2}{p^\star}}\notag\\
           \leq& C(N,p,\Omega )\left(\int_{\Omega }  \left|\nabla (|u|-K_{n+1})^+\right|^{p}\text{d}x\right)^\frac{2}{p}\left|S_{K_{n+1}}\right|^{1-\frac{2}{p^\star}}\notag\\
           \leq& C(N,p,\Omega )\left(\int_{S_{K_{n+1}}}  |\nabla u|^{p}\text{d}x\right)^\frac{2}{p}\left|S_{K_{n+1}}\right|^{1-\frac{2}{p^\star}}.
          \end{align}

Noting $K_{n+1}-K_n=\frac{K}{2^{n+3}}$, we get
          \begin{align*}
           \left(\frac{K}{2^{n+3}}\right)^2|S_{K_{n+1}}|
           =  (K_{n+1}-K_n)^2|S_{K_{n+1}}|
           = \int_{S_{K_{n+1}}}  \left|K_{n+1}-K_n\right|^2\text{d}x
           \leq \int_{S_{K_{n+1}}}  \left((|u|-K_n)^+\right) ^2\text{d}x
           \leq  g_n,
          \end{align*}
           which implies that
           \begin{align}\label{bou7}
               \left|S_{K_{n+1}}\right|\leq&\left(\frac{2^{n+3}}{K}\right)^2g_n.
              \end{align}

By  \eqref{bou7}, we deduce that
              \begin{align}\label{bou8}
               \int_{S_{K_{n+1}}}  ||u|-K_{n+1}|^2\text{d}x
               \leq & 2\int_{S_{K_{n+1}}}  ||u|-K_{n}|^2\text{d}x+2\int_{S_{K_{n+1}}}  |K_n-K_{n+1}|^2\text{d}x\notag\\
               \leq& 2\int_{S_{K_{n}}}  ||u|-K_{n}|^2\text{d}x+2|K_n-K_{n+1}|^2\left|S_{K_{n+1}}\right|\notag\\
               \leq& 2g_n+2\left(\frac{2^{n+3}}{K}\right)^2\left(\frac{K}{2^{n+3}}\right)^2g_n\notag\\
               =& 4g_n.
              \end{align}
Then we have
               \begin{align}\label{bou9}
               K_{n+1}^2 |S_{K_{n+1}}|\leq  \left(\frac{K}{2}\left(1-{\frac{1}{2^{n+2}}}\right)\right)^2\left(\frac{2^{n+3}}{K}\right)^2g_n
               \leq  2^{2n+4}g_n.
              \end{align}

 The inequalities of \eqref{bou5}, {\eqref{bou6}, \eqref{bou7}, \eqref{bou8}, and \eqref{bou9}} imply that
              \begin{align*}
               g_{n+1}
               \leq& C(N,p,\lambda_1,|\lambda_2|,\Omega )\left(\int_{S_{K_{n+1}}}  \left||u|-K_{n+1}\right|^2\text{d}x+K_{n+1}^2 \left|S_{K_{n+1}}\right|\right)^\frac{2}{p}\left|S_{K_{n+1}}\right|^{1-\frac{2}{p^\star}}\notag\\
               \leq& C(N,p,\lambda_1,|\lambda_2|,\Omega )\left(4g_n+2^{2n+4}g_n\right)^\frac{2}{p}\left(\left(\frac{2^{n+3}}{K}\right)^2g_n\right)^{1-\frac{2}{p^\star}}\notag\\
               \leq& C(N,p,\lambda_1,|\lambda_2|,\Omega )\left(4+2^{2n+4}\right)^\frac{2}{p}g_n^\frac{2}{p}\left(\frac{2^{n+3}}{K}\right)^{2-\frac{4}{p^\star}}g_n^{1-\frac{2}{p^\star}}\notag\\
         \leq& C(N,p,\lambda_1,|\lambda_2|,k_0,\Omega )\left(2^{2n+4}\right)g_n^\frac{2}{p}\left({2^{2n+6}}\right)g_n^{1-\frac{2}{p^\star}}\notag\\
               \leq& C(N,p,\lambda_1,|\lambda_2|,k_0,\Omega )16^ng_n^{1+\frac{2}{p}-\frac{2}{p^\star}}.
           \end{align*}		
   Therefore, \eqref{iteration-g} holds true with $C:=C(N,p,\lambda_1,|\lambda_2|,k_0,\Omega ),D:=16$, and $\zeta:=\frac{2}{p}-\frac{2}{p^\star}>0$. In particular, since
               \begin{align*}
                   \left(\frac{K}{4}\right)^{p^\star}\left|S_{k_0}\right|
                   = (K_0)^{p^\star}\left|S_{k_0}\right|
                   \leq \int_{S_{K_0}}  |u|^{p^\star}\text{d}x
                   \leq \int_{\Omega }  |u|^{p^\star}\text{d}x,
               \end{align*}
  it follows that
               \begin{align*}
                   g_0=   \int_{S_{K_0}}  \left((|u|-K_0)^+\right) ^2\text{d}x
                   \leq    \int_{S_{K_0}}  |u|^2\text{d}x
                   \leq   \| u\|_{L^{p^\star}(\Omega )}^{{2}}\left|S_{K_0}\right|^{1-\frac{2}{p^\star}}
                   \leq   \left(\frac{4}{K}\right)^{p^\star-2}  \| u\|_{L^{p^\star}(\Omega )}^{{p^\star}}.
               \end{align*}
   Furthermore, considering a sufficiently large $K$ depending on $\| u\|_{L^{p^\star}(\Omega )}$, it holds that
               \begin{align}\label{g0}
                   {g_0\leq C^{-\frac{1}{\zeta }}D^{-\frac{1}{\zeta^2 }}}.
               \end{align}

    \textbf{Case 2: ${p\in[N,+\infty)}$.}  {Since $N>2$, we} can choose  $q\in (2,N) $.  It is clear that
     {\begin{align*}
     2<q<\min\{p,q^{\star},N\}.
     \end{align*}}

      {By the H\"{o}lder's inequality and the Sobolev embedding theorems, we deduce that}
   \begin{align}\label{bound1}
           g_{n+1}=& \int_{S_{K_{n+1}}}  \left((|u|-K_{n+1})^+\right) ^2\text{d}x\notag\\
           \leq & \left(\int_{S_{K_{n+1}}}  \left((|u|-K_{n+1})^+\right) ^{q^\star}\text{d}x\right)^\frac{2}{q^\star}\left|S_{K_{n+1}}\right|^{1-\frac{2}{q^\star}}\notag\\
           \leq& \left(\int_{\Omega }  \left((|u|-K_{n+1})^+\right) ^{q^\star}\text{d}x\right)^\frac{2}{q^\star}\left|S_{K_{n+1}}\right|^{1-\frac{2}{q^\star}}\notag\\
           \leq& C(N,q,\Omega )\left(\int_{\Omega }  \left|\nabla (|u|-K_{n+1})^+\right|^{q}\text{d}x\right)^\frac{2}{q}\left|S_{K_{n+1}}\right|^{1-\frac{2}{q^\star}}\notag\\
            \leq& C(N,p,\Omega )\left(\int_{\Omega }  \left|\nabla (|u|-K_{n+1})^+\right|^{p}\text{d}x\right)^\frac{2}{p}\left|S_{K_{n+1}}\right|^{\frac{2}{q}-\frac{2}{p}}\left|S_{K_{n+1}}\right|^{1-\frac{2}{q^\star}}\notag\\
            \leq& C(N,p,\Omega )\left(\int_{S_{K_{n+1}}}  |\nabla u|^{p}\text{d}x\right)^\frac{2}{p}\left|S_{K_{n+1}}\right|^{1+\frac{2}{q} -\frac{2}{q^\star}-\frac{2}{p}}.
          \end{align}

 The inequalities of \eqref{bou5}, \eqref{bou7}, \eqref{bou8}, \eqref{bou9}, and  \eqref{bound1}  imply that
           \begin{align*}
              g_{n+1}
              \leq& C(N,p,\lambda_1,|\lambda_2|,\Omega )\left(\int_{S_{K_{n+1}}}  \left||u|-K_{n+1}\right|^2\text{d}x+K_{n+1}^2 \left|S_{K_{n+1}}\right|\right)^\frac{2}{p}\left|S_{K_{n+1}}\right|^{1+\frac{2}{q} -\frac{2}{q^\star}-\frac{2}{p}}\notag\\
              \leq& C(N,p,\lambda_1,|\lambda_2|,\Omega )\left(4g_n+2^{2n+4}g_n\right)^\frac{2}{p}\left(\left(\frac{2^{n+3}}{K}\right)^2g_n\right)^{1+\frac{2}{q} -\frac{2}{q^\star}-\frac{2}{p}}\notag\\
              \leq& C(N,p,\lambda_1,|\lambda_2|,\Omega )\left(4+2^{2n+4}\right)^\frac{2}{p}g_n^\frac{2}{p}\left(\frac{2^{n+3}}{K}\right)^{2\left(1+\frac{2}{q}-\frac{2}{q^\star}-\frac{2}{p}\right)}g_n^{1+\frac{2}{q} -\frac{2}{q^\star}-\frac{2}{p}}\notag\\
          \leq& C(N,p,\lambda_1,|\lambda_2|,k_0,\Omega )\left(2^{2n+4}\right)g_n^\frac{2}{p}\left( {2^{4n+12}} \right)g_n^{1+\frac{2}{q} -\frac{2}{q^\star}-\frac{2}{p}}\notag\\
              \leq& C(N,p,\lambda_1,|\lambda_2|,k_0,\Omega )64^ng_n^{1+\frac{2}{q}-\frac{2}{q^\star}}.
          \end{align*}		
   Therefore, \eqref{iteration-g} holds true with $C:=C(N,p,\lambda_1,|\lambda_2|,k_0,\Omega ),D:=64$, and {$\zeta:=\frac{2}{q}-\frac{2}{q^\star}>0$}. In particular, since
               \begin{align*}
                   \left(\frac{K}{4}\right)^{q^\star}|S_{k_0}|
                   = (K_0)^{q^\star}|S_{k_0}|
                   \leq \int_{S_{K_0}}  |u|^{q^\star}\text{d}x
                   \leq \int_{\Omega }  |u|^{q^\star}\text{d}x,
               \end{align*}
    it follows that
               \begin{align*}
              g_0=   \int_{S_{K_0}}    \left((|u|-K_0)^+\right) ^2\text{d}x
                  \leq   \int_{S_{K_0}}     |u|^2\text{d}x
                   \leq  \| u\|_{L^{q^\star}(\Omega )}^{{2}}\left|S_{K_0}\right|^{1-\frac{2}{q^\star}}
                   \leq   \left(\frac{4}{K}\right)^{q^\star-2}  \| u\|_{L^{q^\star}(\Omega )}^{{q^\star}}.
               \end{align*}
      Furthermore, considering a sufficiently large $K$ depending on $\| u\|_{L^{q^\star}(\Omega )}$, we  infer that \eqref{g0} still holds true.

      We have proved that \eqref{iteration-g}, as well as \eqref{g0}, holds true in two cases. Then by Lemma~\ref{jnn}, we have $g_n\rightarrow 0$ as $n\rightarrow+\infty $. {Therefore, we obtain}
               \begin{align*}
                   \lim_{n\rightarrow +\infty}\int_{S_{K_n}}  \left((|u|-K_n)^+\right) ^2\text{d}x =\int_{S_{{K}/{2}}}  \left(\left(|u|-\frac{K}{2}\right)^+\right)^2\text{d}x.
               \end{align*}
               Thus, we have $|u|\leq \frac{K}{2}$ a.e. in $\Omega $, where $K$ is a positive constant depending on $\| u\|_{L^{p^\star}(\Omega )}$ in Case I, and on $\| u\|_{L^{q^\star}(\Omega )}$ in Case II, respectively. Hence, $K$ depends on $\| u\|_{W^{1,p}(\Omega )}$.

               %


               {\textbf{Step 4: proof of the uniform boundedness of  minimizers.}}
               It suffices to show that $\| u\|_{W^{1,p}(\Omega )}$ depends only on $N$, $p$, $|\lambda_2|$, {$\|g\|_{W^{1,p}(\Omega )}$, and the} diameter of $\Omega$.
               By \eqref{ex5}, we deduce that for any $\varepsilon\in(0,1) $, there exists a positive constant $C$ depending only on $\varepsilon $, $N$, $p$, $|\lambda_2|$, {$\|g\|_{W^{1,p}(\Omega )}$, and} the diameter of $\Omega$ such that
               $\int_{\Omega }  \frac{1}{p}|\nabla u|^p\text{d}x\leq \varepsilon  \int_{\Omega }  |\nabla u|^p\text{d}x+C$, namely,
               $\int_{\Omega } |\nabla u|^p\text{d}x\leq \frac{C}{1-\varepsilon }$, which along with $u-g\in W^{1,p}_0(\Omega )$, implies that
               \begin{align*}
                   \int_{\Omega }  |u|^p\text{d}x
                   \leq&C(p)\int_{\Omega }  | u-g|^p\text{d}x+C(p)\int_{\Omega }  |g|^p\text{d}x\notag\\
                   \leq&C(N,p,\Omega )\int_{\Omega }  |\nabla u-\nabla g|^p\text{d}x+C(p )\int_{\Omega }  |g|^p\text{d}x\notag\\
                   \leq&C(N,p,\Omega )\int_{\Omega } |\nabla u|^p\text{d}x+C(N,p,\Omega )\int_{\Omega }  |\nabla g|^p\text{d}x+C(p)\int_{\Omega } |g|^p\text{d}x\notag\\
                   \leq&{C\left(N,p,|\lambda_2|,\|g\|_{W^{1,p}(\Omega )},\Omega\right)}.
               \end{align*}

               Finally, we get that $\|u\|_{W^{1,p}(\Omega )}$ depends only on $N$, $p$, $|\lambda_2|$, {$\|g\|_{W^{1,p}(\Omega )}$, and} the diameter of $\Omega$.
   \end{proof}
    \section{Local $C^{1,\alpha }$-continuity of   minimizers}\label{sec2'}

{In this section, we prove   local $C^{1,\alpha }$-continuity of  minimizers.}
   \begin{theorem}\label{c1ar}
       Let $u$ be a   minimizer of $\mathcal{J}(u)$.  Then $u\in C^{1,\alpha}_\text{loc}(\Omega )$  {with some $\alpha \in (0,1)$}.
       For any $\Omega'\Subset \Omega $, there exists a positive constant $C$ depending only on $N$, $p$, $\lambda_1$, $|\lambda_2|$, $\| g\| _{{L^\infty}(\Omega )}$, $\| g\| _{W^{1.p}(\Omega )}$,
       $\Omega '$, and {the} diameter of $\Omega $ such that $\| u\| _{C^{1,\alpha}(\Omega ')}\leq C$.
   \end{theorem}
   \begin{proof} Let $B_R=B_R{(x_0)}$ for some $R\leq R_0\leq 1$, where $R_0$ will be chosen later.
   Without loss of generality, let $B_r\Subset B_R \Subset \Omega ${, and} $B_r$ and $B_R$ have the same center.
   Let $v$ be a $p$-harmonic function in $B_R$ satisfying
   \begin{equation*}
       \left\{\begin{aligned}
       & \divi~ \left(|\nabla v|^{p-2}\nabla v\right)=0 && \text{in}\ B_R,\\
       &  v=u&&  \text{on}\ \partial B_R.
       \end{aligned}\right.
       \end{equation*}

By Lemma~\ref{nsigma} and Lemma~\ref{nkappa}, there exists a constant $\delta  \in (0,1)$ such that
       \begin{align}\label{ji7}
      \int_{B_r}  |\nabla u-(\nabla u)_r|^p\text{d}x
       \leq &C(N,p)\left(\frac{r}{R}\right)^{N+\delta }\int_{B_R}  |\nabla u-(\nabla u)_R|^p\text{d}x +C(N,p)\int_{B_R} |\nabla u-\nabla v|^p\text{d}x\notag\\
      \leq & C(N,p)\left(\frac{r}{R}\right)^{N+\delta } \int_{B_R}  |\nabla u-(\nabla u)_R|^p\text{d}x+C\left(\mu,N,p,\|v\|_{L^\infty(B_R)}\right)\int_{B_R}\left(|\nabla u|^p-|\nabla v|^p\right)\text{d}x \notag\\
       &+C\left(\mu,N,p,\|v\|_{L^\infty(B_R)}\right)R^\frac{\mu  }{2}  \left(\int_{B_R}  \left(|\nabla u|^p-|\nabla v|^p\right)\text{d}x\right)^\frac{1}{2},
   \end{align}
       where $\mu  \in(0,N)$ is an arbitrary constant.

       {Note that the maximum principle guarantees that
   \begin{align*}
   \| v\|_{L^\infty(B_R )}\leq \| v\|_{L^\infty(\partial B_R )}=\| u\|_{L^\infty(\partial B_R )}\leq \| u\|_{L^\infty(\Omega )}.
   \end{align*}}  

   In view of the minimality of $u$ and applying the mean value theorem to the function $(1-s^2)^2$, we obtain
    \begin{align}\label{ji1}
         \int_{B_R}  \left(\frac{1}{p}|\nabla u|^p-\frac{1}{p}|\nabla v|^p\right)\text{d}x
       \leq &\int_{B_R}  \left(\lambda_1\left(\left(1-(v^+)^2\right)^2-\left(1-(u^+)^2\right)^2\right)+\lambda_2(v^+-u^+) \right)\text{d}x\notag\\
       = &-4\lambda_1\int_{B_R}  \xi\left( 1-\xi^2 \right)\left( v^+-u^+\right)\text{d}x +\lambda_2\int_{B_R} \left(v^+-u^+\right)\text{d}x\notag\\
       \leq &4\lambda_1\int_{B_R} \left|\xi\right| \left( 1+\xi^2 \right) \left| v^+-u^+\right|\text{d}x +|\lambda_2|\int_{B_R} \left|v^+-u^+\right|\text{d}x\notag\\
   \leq &  4\lambda_1 \| u\| _{L^\infty(\Omega )}\left(1+\| u\| _{L^\infty(\Omega )}^2\right) \int_{B_R} \left| v^+-u^+\right|\text{d}x +|\lambda_2|\int_{B_R} \left|v^+-u^+\right|\text{d}x\notag\\
   \leq &  C \int_{B_R}  \left| v^+-u^+\right|{\text{d}x,}
   \end{align}
 where $\xi \in \left(\min\{u^+,v^+\},\max\{u^+,v^+\}\right)\subset \left[0,\| u\| _{L^\infty(\Omega )}\right]$, {and $C$ depends only on $\lambda_1$, $|\lambda_2|$, and $\|u\|_{L^\infty(\Omega )}$}.

     Note that
   \begin{align}
       \int_{B_R} \left|v^+-u^+\right|\text{d}x
       =&\int_{\{uv\geq 0\}\cap B_R} \left|v^+-u^+\right|\text{d}x+\int_{\{uv< 0\}\cap B_R} \left|v^+-u^+\right|\text{d}x\notag\\
       \leq&\int_{\{uv\geq 0\} \cap B_R} |v -u |\text{d}x+\int_{\{uv<0  \} \cap B_R}|v -u |\text{d}x\notag\\
           =&\int_{B_R} |v-u|\text{d}x\notag\\
       \leq& |B_R|^\frac{1}{N}\cdot \| v-u\|_{L^{1^\star}(B_R)}\notag\\
       \leq& C(N,p,\Omega )R\cdot \int_{B_R}  |\nabla(v-u)|\text{d}x \notag\\
       \leq&  C(N,p,\Omega )\varepsilon R\int_{B_R}  |\nabla v-\nabla u|^p\text{d}x +C(\varepsilon,N,p,\Omega  )R^{1+N},\label{ji4}
      \end{align}
      where $\varepsilon >0$ will be chosen later.
We infer from Lemma~\ref{nkappa},  \eqref{ji1}, and  \eqref{ji4} that
      \begin{align*}
        \int_{B_R}  |\nabla u-\nabla v|^p\text{d}x
       \leq& C\int_{B_R}  \left(|\nabla u|^p-|\nabla v|^p\right)\text{d}x +CR^\frac{\mu }{2}\left(\int_{B_R}  \left(|\nabla u|^p-|\nabla v|^p\right)\text{d}x \right)^\frac{1}{2}\notag\\
       \leq& C(\varepsilon)R^{N+1}+C\varepsilon R\int_{B_R} |\nabla u-\nabla v|^p\text{d}x+C(\varepsilon)R^\frac{\mu  +1+N}{2}+C\varepsilon^\frac{1}{2} R^\frac{\mu  +1}{2}\left(\int_{B_R} |\nabla u-\nabla v|^p\text{d}x\right)^\frac{1}{2}\notag\\
       \leq& C(\varepsilon)R^{N+1}  +C\varepsilon R\int_{B_R}  |\nabla u-\nabla v|^p\text{d}x+C(\varepsilon)R^\frac{\mu  +1+N}{2}  +\varepsilon \int_{B_R} |\nabla u-\nabla v|^p\text{d}x+CR^{\mu  +1}.
      \end{align*}

Considering a sufficiently small $\varepsilon$, we get
      \begin{align}\label{ji5}
      \int_{B_R}  |\nabla u-\nabla v|^p\text{d}x\leq C\left(\mu,N,p,\lambda_1 ,|\lambda_2|,\|g\|_{W^{1,p}(\Omega )},\|g\|_{L^\infty(\Omega )} ,\|v\|_{L^\infty(B_R)},\Omega\right)R^m,
      \end{align}
      where $m:=\min\left\{N+1,\ \mu  +1,\ \frac{\mu  +1+N}{2}\right\}$.

By \eqref{ji1}, {\eqref{ji4}, and} \eqref{ji5}, we have
      \begin{align}\label{ji6}
       \int_{B_R}  \left(|\nabla u|^p-|\nabla v|^p\right)\text{d}x\leq CR^{m+1}+CR^{N+1}.
       \end{align}

 Putting \eqref{ji6} into \eqref{ji7}, we get for {all $r\in(0,R]$ {that}}
       \begin{align*}
             \int_{B_r} \left|\nabla u-(\nabla u)_r\right|^p\text{d}x
            \leq & C\left(\frac{r}{R}\right)^{N+\delta }\int_{B_R} |\nabla u-(\nabla u)_R|^p\text{d}x +C\int_{B_R}  \left(|\nabla u|^p-|\nabla v|^p\right)\text{d}x\notag\\
            & +CR^\frac{\mu  }{2}  \left(\int_{B_R}  \left(|\nabla u|^p-|\nabla v|^p\right)\text{d}x\right)^\frac{1}{2}\notag\\
            \leq& C\left(\frac{r}{R}\right)^{N+\delta }\int_{B_R}  |\nabla u-(\nabla u)_R|^p\text{d}x +CR^{N+1}+CR^{m+1}
            +CR^\frac{\mu  +N+1}{2}+CR^\frac{\mu  +m+1}{2}.
           \end{align*}
Since $\mu  \in(0,N)$ is an arbitrary constant, {we can choose $\mu  \in (N-1,N)$.  Therefore,} $\min\left\{m+1,\ \frac{\mu + N+1}{2},\ \frac{\mu + m+1}{2}\right\}>N$.

        We deduce that there exists a constant $\alpha _1>0$ such that
        \begin{align*}
           \int_{B_r} |\nabla u-(\nabla u)_r|^p\text{d}x
            \leq & C\left(\frac{r}{R}\right)^{N+\delta }\int_{B_R}  |\nabla u-(\nabla u)_R|^p\text{d}x +CR^{N+\alpha _1}.
           \end{align*}

 By the inequality (5.1) in \cite{lieberman1991natural}, it holds that
           \begin{align*}
             \int_{B_r}  \left|\nabla u-(\nabla u)_r\right|^p\text{d}x
                \leq & C(N,p)\int_{B_r} |\nabla u-(\nabla u)_R|^p\text{d}x
                \leq C(N,p)\int_{B_R}  |\nabla u-(\nabla u)_R|^p\text{d}x.
               \end{align*}

 Applying Lemma~\ref{phir} with $\tau  (r):=\int_{B_r}  |\nabla u-(\nabla u)_r|^p\text{d}x$,  we deduce that  there exists a constant $\alpha _2\in(0,1)$ such that
               \begin{align*}
                   \int_{B_r}  |\nabla u-(\nabla u)_r|^p\text{d}x \leq {\left(N,p,\lambda_1 ,|\lambda_2|,\|g\|_{W^{1,p}(\Omega )},\|g\|_{L^\infty(\Omega )} ,\|v\|_{L^\infty(B_R)},\Omega\right)}r^{N+\alpha _2},
                   \end{align*}
          {which yields}
           \begin{align*}
                   \int_{B_r}  |\nabla u-(\nabla u)_r|\text{d}x
            \leq&  \left(\int_{B_r}  |\nabla u-(\nabla u)_r|^p\text{d}x \right)^{\frac{1}{p}}|B_r|^{\frac{p-1}{p}}\\
            \leq  &{{C\left(N,p,\lambda_1 ,|\lambda_2|,\|g\|_{W^{1,p}(\Omega )},\|g\|_{L^\infty(\Omega )} ,\|v\|_{L^\infty(B_R)},\Omega\right)}r^{N+\alpha}}
               \end{align*}
            {with $\alpha:=\frac{\alpha_2}{p} \in (0,1)$.} Hence, $\nabla u$ is in the Campanato space. Using the Campanato's embedding theorem, we {obtain the local $C^{1,\alpha}$-continuity of  $ u$.}
         \end{proof}

         \section{Optimal growth  of minimizers near the free boundary}\label{sec3}

         In this section, we prove the optimal growth  for  minimizers of the functional $\mathcal{J}(u)$ near the free boundary.

         \begin{theorem}\label{the31}
         Let $u$ be a   minimizer of $\mathcal{J} (u)$ and assume that $\Gamma^+\neq \emptyset$.  Let $y\in \Gamma^+$ and $B_{r_0}(y)\Subset \Omega $ with some $r_0>0$.
         Then, there exist positive constants $C_0$ and $C_1$ depending only on $r_0$, $N$, $p$, $\lambda_1$, $|\lambda_2|$, $\| g\| _{{L^\infty}(\Omega )}$, $\| g\| _{W^{1.p}(\Omega )}$,
         and {the} diameter of $\Omega $ such that
         \begin{subequations}
         \begin{align}
         &| u(x)| \leq C_0| x-y|^\frac{p}{p-1}, \   x\in B_r(y),\label{u1}\\
         &| \nabla u(x)| \leq C_1| x-y|^\frac{1}{p-1}, \  x\in B_r(y) \label{u2}
         \end{align}
         \end{subequations}
         hold true for all $r<r_0$.
         \end{theorem}
         \begin{proof}
         Since local properties of the free boundary are considered in this paper, we may assume that $u$ is a non-negative minimizer  {defined} in the ball $B_1(y)$ with $y=0\in \Gamma^+$.

        We prove \eqref{u1}. {Define $S(j,u):=\displaystyle \sup_{x\in B_{2^{-j}}} | u(x)|.$}
         It suffices to show that  for all $j\in \mathbb{N} ^+$ there holds
         \begin{align}\label{sj}
         S(j+1,u)\leq \max &\left\{c \left(2^{-j}\right)^\frac{p}{p-1}, S(j,u) \left(2^{-1}\right)^\frac{p}{p-1},..., S(j-m,u)\left(2^{-(m+1)}\right)^\frac{p}{p-1},..., S(0,u)\left(2^{-j-1}\right)^\frac{p}{p-1}\right\}
         \end{align}
         with some $c>0$.

         We prove by contradiction. Suppose that  {\eqref{sj}} does not hold. Then for all $k\in \mathbb{N} ^+$, there exist a sequence of integers $j_k\in \mathbb{N}^+$ and a sequence of minimizers $u_k$ such that
         \begin{align}\label{sjj}
         S(j_k+1,u_k)> \max &\left\{k\left(2^{-j_k}\right)^\frac{p}{p-1}, S(j_k,u_k)\left(2^{-1}\right)^\frac{p}{p-1},..., S\left(j_k-i,u_k\right)\left(2^{-(i+1)}\right)^\frac{p}{p-1},..., S(0,u_k) \left(2^{-j_k-1}\right)^\frac{p}{p-1}\right\}.
         \end{align}

         {We infer from the boundedness of $u_k$  and \eqref{sjj}  that $j_k\rightarrow \infty$ as $k\rightarrow \infty$}.
          For the sake of discussion, we define $\left(1-(u^+)^2\right)^2:=1-F_1(u^+)+F_2(u^+)$ {with  $F_1(t):=2t^2$ and $F_2(t):=t^4$ for $t\in \mathbb{R}$}. Let $v_k(x):=\frac{u_k\left(2^{-j_k}x\right)}{S(j_k+1,u_k)}$, $ \rho_k:=2^{j_k}S(j_k+1,u_k)$,
         {$F_{1,k}(t):=\frac{2S^2(j_k+1,u_k) }{\rho_k^p}t^2$, $F_{2,k}(t):=\frac{S^4(j_k+1,u_k) }{\rho_k^p}t^4$}, $\lambda _{1,k}:=\lambda_1 $, and $\lambda_{2,k}:=\frac{S(j_k+1,u_k) }{\rho_k^p}\lambda_2$.

         By direct computations, for any $k>0$, it follows that
         \begin{align*}
         &\int_{B_{2^{j_k}}}  \left(\frac{1}{p}| \nabla v_k| ^p+\lambda_{1,k}\left(-F_{1,k}\left(v_k^+\right)+F_{2,k}\left(v_k^+\right)\right)+\lambda_{2,k}v^+_k\right)\text{d}x\\
         =&\int_{B_{2^{j_k}}}  \frac{1}{\rho ^p_k}\left(\frac{1}{p}\left| \nabla u_k\left(2^{-j_k}x\right)\right| ^p+\lambda_1\left(-F_1\left(\left(u_k\left(2^{-j_k}x\right)\right)^+\right)+F_2\left(\left(u_k\left(2^{-j_k}x\right)\right)^+\right)\right)+\lambda_2\left(u_k\left(2^{-j_k}x\right)\right)^+\right)\text{d}x\\
         =&\frac{2^{j_kN}}{\rho ^p_k}\int_{B_1}  \left(\frac{1}{p}| \nabla u| ^p+\lambda_1\left(-F_1(u^+)+F_2(u^+)\right)+\lambda_2u^+\right)\text{d}x.
         \end{align*}
         Therefore, $v_k$ is a non-negative minimizer of the functional
       \begin{align*}
         \int_{B_R}\left(\frac{1}{p}| \nabla v| ^p+\lambda_{1,k}\left(1-F_{1,k}(v^+)+F_{2,k}(v^+)\right)+\lambda_{2,k}v^+\right)\text{d}x,
         \end{align*}
where $R:=2^{i_0}<2^{j_k}$ and $i_0\geq 0$ is a fixed number.

         {In view of \eqref{sjj}, we have $\displaystyle \sup_{B_R} |v_k|\leq (2R)^\frac{p}{p-1}$ and $S(j_k+1,u_k)\geq k\left(2^{-j_k}\right)^\frac{p}{p-1}$, which implies}
         \begin{align*}
         &2^{j_kp}S^{p-1}(j_k+1,u_k)\geq 2^{j_kp}k^{p-1}2^{-j_k(p-1)\frac{p}{p-1}}= k^{p-1}.
         \end{align*}

 Letting $k\rightarrow \infty$, we infer from $\displaystyle \sup_{B_R} |v_k|\leq (2R)^\frac{p}{p-1}$ and Theorem~\ref{t1} that
         \begin{subequations}\begin{align}
         &|\lambda_{2,k}|=\frac{S(j_k+1,u_k)|\lambda_2|}{\rho_k^p}=\frac{|\lambda_2|}{2^{j_kp}S^{p-1}(j_k+1,u_k)}\leq \frac{|\lambda_2|}{k^{p-1}}\rightarrow 0,\label{H1}\\
         & \left|F_{1,k}(v_k^+)\right|=\frac{2S^2(j_k+1,u_k)(v_k^+)^2}{\rho_k^p}\leq \frac{2S(j_k+1,u_k)(2R)^\frac{2p}{p-1}}{2^{j_kp}S^{p-1}(j_k+1,u_k)}\leq \frac{C}{k^{p-1}}\rightarrow 0,\label{F1}\\
         &\left|F_{2,k}(v_k^+)\right|=\frac{S^4(j_k+1,u_k)(v_k^+)^4}{\rho_k^p}\leq \frac{S^3(j_k+1,u_k)(2R)^\frac{4p}{p-1}}{2^{j_kp}S^{p-1}(j_k+1,u_k)}\leq \frac{C}{k^{p-1}}\rightarrow 0,\label{F2}
         \end{align}\end{subequations}
         where both $C$ in \eqref{F1} and \eqref{F2} are positive constants that are independent of $k$.

         For $k$ large enough, by the local $C^{1,\alpha }$-regularity of minimizers, we have $\| v_k\| _{C^{1,\alpha }(B_R)}\leq C$, where {$C$ is a positive constant independent of $k$ as $k\rightarrow \infty$ due to \eqref{H1}, \eqref{F1}, and \eqref{F2}}. Then up to a subsequence, we get $v_k\rightarrow v_0$ in $C^{1,\beta }(\overline{B}_r)$ {for any $r<R$ and $ \beta \in (0,\alpha) $.}
         From $v_k(0)=\frac{u_k(0)}{S(j_k+1,u_k)}=0$ and $\displaystyle \sup_{{B_{1/2}}}| v_k|={\displaystyle \sup_{B_{1/2}}u_k\left(2^{-j_k}x\right)}/ {S(j_k+1,u_k)} =1$, we infer that
         \begin{align}\label{v0}
         \displaystyle \sup_{B_{{1}/{2}}}| v_0|=1,\ v_0(0)=0.
         \end{align}

         Furthermore, by the minimality of $v_k$, for any $\varphi \in C_0^\infty(B_r)$, we have
         \begin{align*}
         &\int_{B_r}  \frac{1}{p}|\nabla v_k|^p\text{d}x\\
         \leq & \int_{B_r}   \bigg(\frac{1}{p}\left| \nabla v_k+\nabla \varphi \right| ^p+\lambda_{1,k}\left(F_{1,k}\left(v_k^+\right)-F_{1,k}\left((v_k+\varphi )^+\right)+F_{2,k}\left((v_k+\varphi )^+\right)-F_{2,k}\left(v_k^+\right)\right)+\lambda_{2,k}\left((v_k+\varphi)^+ -v_k^+\right)\bigg) \text{d}x.
         \end{align*}

        {Note} that {$v_k+\varphi$  is uniformly bounded in $B_r$}.
         Based on \eqref{H1}, {\eqref{F1}, and} \eqref{F2}, {for} $k\rightarrow \infty$, we derive
         \begin{align*}
       \int_{B_r}  \left| F_{1,k}\left(v_k^+\right)-F_{1,k}\left((v_k+\varphi )^+\right)\right| \text{d}x
         \leq &\int_{B_r}  \left(\left|F_{1,k}\left(v_k^+\right)\right|+\left|F_{1,k}\left((v_k+\varphi)^+\right)\right|\right) \text{d}x\\
         \leq&   \int_{B_r}  \left(\frac{C}{k^{p-1}}+\frac{2S^2(j_k+1,u_k)\left((v_k+\varphi )^+\right)^2}{\rho_k^p}\right)\text{d}x\\
         \leq &\int_{B_r}  \left(\frac{C}{k^{p-1}}+\frac{2S(j_k+1,u_k)C}{2^{j_kp}S^{p-1}(j_k+1,u_k)}\right)\text{d}x \\
         \leq &\int_{B_r}  \left(\frac{C}{k^{p-1}}+\frac{C}{k^{p-1}}\right)\text{d}x\rightarrow 0,\\
           \int_{B_r}  \left| F_{2,k}\left((v_k+\varphi)^+ \right)-F_{2,k}\left(v_k^+\right)\right| \text{d}x
             \leq &\int_{B_r}  \left(\left| F_{2,k}\left(v_k^+\right)\right|+\left|F_{2,k}\left((v_k+\varphi)^+\right)\right| \right)\text{d}x\\
         \leq& \int_{B_r}  \left(\frac{C}{k^{p-1}}+\frac{S^4(j_k+1,u_k)\left((v_k+\varphi )^+\right)^4}{\rho_k^p}\right)\text{d}x\\
         \leq &\int_{B_r}  \left(\frac{C}{k^{p-1}}+\frac{S^3(j_k+1,u_k)C}{2^{j_kp}S^{p-1}(j_k+1,u_k)}\right)\text{d}x \\
         \leq &\int_{B_r}  \left(\frac{C}{k^{p-1}}+\frac{C}{k^{p-1}}\right)\text{d}x\rightarrow 0,
         \end{align*}
         and	
         \begin{align*}\int_{B_r} | \lambda_{2,k}| \left| (v_k+\varphi)^+-v_k^+ \right| \text{d}x \leq &\int_{B_r}  \left(\frac{|\lambda_2|\left((v_k+\varphi)^++v_k^+ \right)}{k^{p-1}}\right)\text{d}x
         \leq  \int_{B_r}  \frac{C|\lambda_2|}{k^{p-1}}\text{d}x \rightarrow 0,
         \end{align*}
         where $C$ appearing in the above inequalities is {a positive constant, which is   independent} of $k$.
          Therefore, we have $\int_{B_r}  \frac{1}{p}|\nabla v_k|^p\text{d}x\leq \int_{B_r}  \frac{1}{p}|\nabla v_k+\nabla \varphi |^p\text{d}x$.
         Due to  $v_k\rightarrow v_0$ in $C^{1,\beta}(\overline {B}_r)$, we {infer that}
         \begin{align*}
             \int_{B_r}  \frac{1}{p}|\nabla v_0|^p\text{d}x\leq \int_{B_r}  \frac{1}{p}|\nabla v_0+\nabla \varphi |^p\text{d}x.
         \end{align*}	
         Thus, $v_0$ is a $p$-harmonic function in $B_r$.
         Note that $v_k\geq 0$ in $B_1$ and $v_0\geq 0$ in $B_r$. By {virtue of \eqref{v0} and using   the Harnack's inequality, we} obtain $v_0\equiv 0$ in $B_r$.
         Due to the continuity of $v_0$ and the arbitrariness of $r$, we have $v_0\equiv 0$ in $B_R$, which is a contradiction with \eqref{v0}. Thus, \eqref{u1} holds true.

         Now we prove \eqref{u2}.
         Let $S(j,|\nabla u|):=\displaystyle \sup_{x\in B_{2^{-j}}} | \nabla u(x)|$. It suffices to show {that} for all $j\in \mathbb{N} ^+$ there holds
         \begin{align}\label{sj2}
            S(j+1,|\nabla u|)
            \leq &\max \Big\{c\left(2^{-j}\right)^\frac{1}{p-1}, S(j,|\nabla u|)\left(2^{-1}\right)^\frac{1}{p-1},..., S(j-m,|\nabla u|)\left(2^{-(m+1)}\right)^\frac{1}{p-1},..., S(0,|\nabla u|)\left(2^{-j-1}\right)^\frac{1}{p-1}\Big\}
         \end{align}
         with some $c>0$.

         We prove by contradiction. Suppose that {\eqref{sj2}} does not hold. Then for all $k\in \mathbb{N} ^+$, there exist a sequence of integers $j_k\in \mathbb{N}^+$ and a sequence of minimizers $u_k$ such that
         \begin{align}\label{sjj2}
         S(j_k+1,|\nabla u_k|)
         >  \max  \Big\{&k \left(2^{-j_k}\right)^\frac{1}{p-1}, S(j_k,|\nabla u_k|) (2^{-1})^\frac{1}{p-1}, ...,S(j_k-i,|\nabla u_k|)\left(2^{-(i+1)}\right)^\frac{1}{p-1},...,\notag\\
          &S(0,|\nabla u_k|)\left(2^{-j_k-1}\right)^\frac{1}{p-1}\Big\}.
         \end{align}

Let $v_k(x):=\frac{u_k\left(2^{-j_k}x\right)}{2^{-j_k}S(j_k+1,|\nabla u_k|)}$, $\rho_k:=S(j_k+1,|\nabla  u_k|)$,
         $F_{1,k}(t):=\frac{2^{-\left(2j_k-1\right)}S^2(j_k+1,|\nabla u_k|) }{\rho_k^p}t^2$, $F_{2,k}(t):=\frac{2^{-4j_k}S^4(j_k+1,|\nabla u_k|)}{\rho_k^p}t^4$, $\lambda _{1,k}:=\lambda_1 $, and $\lambda_{2,k}:=\frac{2^{-j_k}S(j_k+1,|\nabla u_k|) }{\rho_k^p}\lambda_2$.

         By direct computations, for any $k>0$, it {holds} that
         \begin{align*}
         &\int_{B_{2^{j_k}}}  \left(\frac{1}{p}| \nabla v_k(x)| ^p+\lambda_{1,k}\left(-F_{1,k}\left(v_k^+\right)+F_{2,k}\left(v_k^+\right)\right)+\lambda_{2,k}v_k^+\right)\text{d}x\\
         =&\int_{B_{2^{j_k}}}  \frac{1}{\rho ^p_k}\left(\frac{1}{p}\left| \nabla u_k\left(2^{-j_k}x\right)\right| ^p+\lambda_1\left(-F_1\left(\left(u_k\left(2^{-j_k}x\right)\right)^+\right)+F_2\left(\left(u_k\left(2^{-j_k}x\right)\right)^+\right)\right) +\lambda_2\left(u_k\left(2^{-j_k}x\right)\right)^+\right)\text{d}x\\
         =&\frac{2^{j_kN}}{\rho ^p_k}\int_{B_1}  \left(\frac{1}{p}| \nabla u| ^p+\lambda_1\left(-F_1(u^+)+F_2(u^+)\right)+\lambda_2u^+\right)\text{d}x.
         \end{align*}
         Therefore, $v_k$ is a non-negative minimizer of the functional
         {\begin{align*}
             \int_{B_{R}}  \left(\frac{1}{p}| \nabla v| ^p+\lambda_{1,k}\left(1-F_{1,k}(v^+)+F_{2,k}(v^+)\right)+\lambda_{2,k}v^+\right)\text{d}x,
         \end{align*}}
         where $R:=2^{i_0}<2^{j_k}$ and $i_0\geq 0$ is a fixed number.

         By \eqref{u1} and \eqref{sjj2}, we have $\displaystyle \sup_{B_1} |v_k|\leq \frac{C_0}{k}\rightarrow 0$ as $k\rightarrow \infty$ and $S(j_k+1,|\nabla u_k|)\geq k\left(2^{-j_k}\right)^\frac{1}{p-1},$ {among which} the latter implies
\begin{align*}
         &2^{j_k}S^{p-1}(j_k+1,|\nabla u_k|)\geq 2^{j_k}k^{p-1}2^{-j_k}=k^{p-1}.
         \end{align*}
Then,
 {for} $k\rightarrow \infty$, we have
         \begin{subequations}
         \begin{align}
          &|\lambda_{2,k}|=\frac{2^{-j_k}S(j_k+1,|\nabla u_k|)|\lambda_2|}{\rho_k^p}
         = \frac{|\lambda_2|}{2^{j_k}S^{p-1}(j_k+1,|\nabla u_k|)}
       \leq   \frac{|\lambda_2|}{k^{p-1}}
         \rightarrow  0,\label{H12}\\
       &  \left|F_{1,k}(v_k^+)\right|   =\frac{2^{-\left(2j_k-1\right)}S^2(j_k+1,|\nabla u_k|)\left(v_k^+\right)^2}{\rho_k^p}
         \leq   \frac{2S(j_k+1,|\nabla u_k|)C_0^2}{k^22^{2j_k}S^{p-1}(j_k+1,|\nabla u_k|)}
         \leq   \frac{C}{k^{p+1}}
       \rightarrow  0, \label{F12}\\
        &\left|F_{2,k}(v_k^+)\right|=\frac{2^{-4j_k}S^4(j_k+1,|\nabla u_k|)\left(v_k^+\right)^4}{\rho_k^p}
       \leq   \frac{S^3(j_k+1,|\nabla u_k|)C_0^4}{k^42^{4j_k}S^{p-1}(j_k+1,|\nabla u_k|)}
       \leq  \frac{C}{k^{p+3}}
       \rightarrow  0,\label{F22}
         \end{align}
         \end{subequations}
         where {$C$ in \eqref{F12} and \eqref{F22} is a positive constant  independent of $k$.}

         Then, we get $v_k\rightarrow v_0$ in $C^{1,\beta }(\overline{B}_r)$ with any $r<R$ and {some $ \beta \in (0,1) $}.
         We deduce by $v_k(0)=\frac{u_k(0)}{S(j_k+1,u_k)}=0$ and  {$\displaystyle \sup_{B_{{1}/{2}}}| \nabla v_k|=\displaystyle \sup_{B_{1/2}}\frac{\left|\nabla u_k\left(2^{-j_k}x\right)\right|}{S(j_k+1,|\nabla u_k|)} =1$} that
         \begin{align}\label{v02}
         \displaystyle \sup_{B_{{1}/{2}}}| \nabla v_0|=1,\ v_0(0)=0.
         \end{align}
 Furthermore, by the minimality of $v_k$, for any $\varphi \in C_0^\infty(B_r)$, {it holds that}
         \begin{align*}
         &\int_{B_r}  \frac{1}{p}|\nabla v_k|^p\text{d}x\notag\\
         \leq & \int_{B_r}  \bigg(\frac{1}{p}| \nabla v_k+\nabla \varphi | ^p+\lambda_{1,k}\left(F_{1,k}\left(v_k^+\right)-F_{1,k}\left((v_k+\varphi)^+\right) +F_{2,k}\left((v_k+\varphi)^+\right)-F_{2,k}\left(v_k^+\right)\right)+\lambda_{2,k}\left((v_k+\varphi)^+ -v_k^+\right)\bigg)\text{d}x.
         \end{align*}
 {{Note} that {$v_k+\varphi$  is uniformly bounded in $B_r$}.}
         Based on \eqref{H12}, {\eqref{F12}, and \eqref{F22}}, {for} $k\rightarrow \infty$, we {have}
         \begin{align*}
        \int_{B_r}   \left|F_{1,k}(v_k)-F_{1,k}(v_k+\varphi )\right| \text{d}x
         \leq &\int_{B_r}  \left|F_{1,k}\left(v_k^+\right)\right|+\left|F_{1,k}\left((v_k+\varphi)^+\right)\right| \text{d}x\\
         \leq& \int_{B_r} \left( \frac{C}{k^{p+1}}+\frac{2^{-\left(2j_k-1\right)}S^2(j_k+1,|\nabla u_k|)\left((v_k+\varphi )^+\right)^2}{\rho_k^p}\right) \text{d}x\\
         \leq &\int_{B_r}  \left(\frac{C}{k^{p+1}}+\frac{2S(j_k+1,|\nabla u_k|)C }{2^{2j_k}S^{p-1}(j_k+1,|\nabla u_k|)}\right)\text{d}x \\
         \leq &\int_{B_r}  \left(\frac{C}{k^{p+1}}+\frac{C}{k^{p-1}}\right)\text{d}x
         \rightarrow  0,\\
          \int_{B_r}  \left| F_{2,k}\left((v_k+\varphi )^+\right)-F_{2,k}(v_k^+)\right| \text{d}x
             \leq &\int_{B_r}  \left| F_{2,k}\left(v_k ^+\right)\right|+\left|F_{2,k}\left((v_k+\varphi)^+\right)\right| \text{d}x\\
         \leq& \int_{B_r}  \left(\frac{C}{k^{p+3}}+\frac{2^{-4j_k}S^4(j_k+1,|\nabla u_k|)\left((v_k+\varphi )^+\right)^4}{\rho_k^p}\right)\text{d}x\\
         \leq &{\int_{B_r}  \left(\frac{C}{k^{p+3}}+\frac{CS^3(j_k+1,|\nabla u_k|)}{2^{4j_k}S^{p-1}(j_k+1,|\nabla u_k|)}\right)\text{d}x} \\
         \leq &\int_{B_r}  \left(\frac{C}{k^{p+3}}+\frac{C}{k^{p-1}}\right)\text{d}x
         \rightarrow  0,
         \end{align*}
         and
         \begin{align*}
             \int_{B_r}  | \lambda_{2,k}| \left| (v_k+\varphi)^+-v_k^+ \right| \text{d}x \leq &\int_{B_r}  \frac{|\lambda_2|\left((v_k+\varphi)^++v_k^+\right)}{k^{p-1}}\text{d}x
         \leq\int_{B_r}  \frac{C|\lambda_2|}{k^{p-1}}\text{d}x  \rightarrow 0,
         \end{align*}
         where $C$ appearing in the {above inequalities} {is} independent of $k$.
 Therefore, we have $\int_{B_r}  \frac{1}{p}|\nabla v_k|^p\text{d}x\leq \int_{B_r}  \frac{1}{p}|\nabla v_k+\nabla \varphi |^p\text{d}x$.
         Due to  $v_k\rightarrow v_0$ in $C^{1,\beta}(\overline {B}_r)$, we {infer that}
         \begin{align*}
             \int_{B_r}  \frac{1}{p}|\nabla v_0|^p\text{d}x\leq \int_{B_r}  \frac{1}{p}|\nabla v_0+\nabla \varphi |^p\text{d}x.
         \end{align*}
         Thus, $v_0$ is a $p$-harmonic function in $B_r$.
         Note that $v_k\geq 0$ in $B_1$ and $v_0\geq 0$ in $B_r$. {In view of \eqref{v02} and utilizing the Harnack's inequality}, 
         we obtain $v_0\equiv 0$ in $B_r$.
          Due to the continuity of $v_0$ and the arbitrariness of $r$, we have $v_0\equiv 0$ in $B_R$ and $\displaystyle\sup_{B_{1/2}}|\nabla v_0|=0$, which a contradiction with \eqref{v02}. Thus, \eqref{u2} holds true.
         \end{proof}

 \section{Non-degeneracy of minimizers when $\lambda_2>0$}\label{sec3'}
        In this section, we prove the non-degeneracy property of minimizers  of $\mathcal{J} (u)$ when $\lambda_2>0$}. First, for  general $\lambda_2\in\mathbb{R}$, we show  that any minimizer of $\mathcal{J} (u)$   is a solution to a Ginzburg-Landau equation. More specifically, the following result holds true, whose proof can be  {proceeded} in a similar way as \cite[Lemma 4.2]{zheng2022free}  and hence  is omitted.

         \begin{lemma}
         Let $u$ be a   minimizer of $\mathcal{J} (u)$ with    $\lambda_2\in\mathbb{R}$. Then $u$ is a weak solution {to} the Ginzburg-Landau equation:
         \begin{equation}\label{euler}
         {\divi~\left(| \nabla u|^{p-2}\nabla u\right)=4\lambda_1 \left(-u + u  ^3\right)+\lambda_2 \ \text{in}~\{u>0\}.}
         \end{equation}
         \end{lemma}
         Now, for $\lambda_2>0$, we   prove    non-degeneracy for  minimizers of the functional $\mathcal{J} (u)$ near the free boundary $\Gamma ^+$.
         \begin{proposition}\label{the32}
         Let $u$ be a   minimizer of $\mathcal{J} (u)$ with  {$\lambda_2>0$}. Then for every $ y\in \Gamma^+$ and $B_{r_1}\Subset \Omega $ with some sufficiently small positive  {constant}
         $r_1\leq\text{min}\left\{r_0,\left(\frac{\lambda_2}{16\lambda_1 C_0}\right)^{\frac{p-1}{p}}\right\},$
         there exists a positive constant $C_2$ depending only on {$r_1$, $N$, $p$, and $\lambda_2$} such that
         \begin{align*}
         \displaystyle \sup_{\partial B_r(y)\cap  \{u>0\}}u \geq {C_2  r^{\frac{p}{p-1}}}
         \end{align*}
         holds true for all {$r\in \left(0,\frac{r_1}{2}\right)$},  where $r_0$ and $C_0$ are positive constants {determined by} Theorem~\ref{the31}.
         \end{proposition}

         \begin{proof}
         Let $y\in  \{u>0\}$. For all $x_0\in \Gamma ^+$ satisfying $|y-x_0|\leq \frac{r_1}{2}$, we have $B_r(y)\subset B_{r_1}({x_0}) $ with $r< \frac{r_1}{2}$.
         Define $v(x):=C_2| x-y| ^{\frac{p}{p-1}}$ {in $B_r(y)$}, where $C_2:=\left(\frac{3\lambda_2}{4N}\right)^{\frac{1}{p-1}}\frac{p-1}{p}$.
         By direct computations, we have
         $\nabla v=C_2 \frac{p}{p-1} |x-y|^{\frac{2-p}{p-1}}(x-y)$.
         It follows that
         \begin{align*}
         {\divi~\left(| \nabla v|^{p-2}\nabla v\right)}
         =&\divi~ \left(C_2^{p-2}\left(\frac{p}{p-1}\right)^{p-2}|x-y|^{\frac{p-2}{p-1}}\cdot C_2\frac{p}{p-1}|x-y|^{\frac{2-p}{p-1}}(x-y) \right)\\
         =&C_2^{p-1} \left(\frac{p}{p-1}\right)^{p-1} \divi~(x-y)\\
         =&NC_2^{p-1} \left(\frac{p}{p-1}\right)^{p-1}\   \text{in}\  B_r(y).
         \end{align*}

We deduce by \eqref{euler} that
         \begin{align*}
         {\divi~\left(| \nabla u|^{p-2}\nabla u\right)} =4\lambda_1 {\left(-u+u^3\right)+ \lambda_2} \  {\text{in}\  B_r(y)\cap  \{u>0\}}.
         \end{align*}

         {Noting  that $B_r(y)\subset B_{r_1}({x_0}) $ and applying Theorem~\ref{the31}}, we obtain $|u|\leq C_0r_1^{\frac{p}{p-1}}\leq C_0\frac{\lambda_2}{16\lambda_1 C_0}\leq\frac{\lambda_2}{16\lambda_1 } $ in $B_{r_1}$ for $\lambda_1>0$. Then, we have $4\lambda _1{\left(-u+u^3\right)+\lambda_2}> \frac{3\lambda_2}{4}$ in $B_{r_1}\cap\{u>0\}$ for $\lambda_1>0$. In addition, it is clear that $4\lambda _1{\left(-u+u^3\right)+\lambda_2}> \frac{3\lambda_2}{4}$ in $B_{r_1}$  for  $\lambda_1= 0$. Therefore, for any {$\lambda_1\geq 0$ and $\lambda_2 > 0$}, it always holds that
         \begin{align*}
          \divi~\left(| \nabla u|^{p-2}\nabla u\right) \geq   \frac{3}{4}\lambda_2
          \geq  NC_2^{p-1}\left(\frac{p}{p-1}\right)^{p-1}
          =  \divi~\left(| \nabla v|^{p-2}\nabla v\right) \ \text{in}\ B_r(y)\cap \{u>0\}.
         \end{align*}

         By the definitions of {$u $ and $v $}, we have
         \begin{align*}
          {u =0\leq v}  \ \text{on}\ B_r(y)\cap \Gamma ^+.
         \end{align*}

 If {$u \leq v $} on $\partial B_r(y)\cap  \{u>0\}$. Then by the comparison principle (see \cite[Lemma 4.1]{tan2013orlicz}), we have
         \begin{align*}
          {u \leq v } \ \text{in}\  B_r(y)\cap  \{u>0\}.
         \end{align*}
         However, $u(y)>0=v(y)$, which is a contradiction. Therefore, there exists $y_0\in \partial B_r(y)\cap  \{u>0\}$ such that
         \begin{align*}
         u(y_0)\geq v(y_0)=C_2r^{\frac{p}{p-1}},
         \end{align*}
         {which leads to}
         \begin{align*}
         \displaystyle \sup_{\partial B_r(y)\cap  \{u>0\}} u\geq C_2 r^{\frac{p}{p-1}}.
         \end{align*}

         {For $y\in \partial  \{u>0\}$,} we  take first $y_i\in  \{u>0\}$ such that $y_i\rightarrow y$, then, proceeding as above, we can show that
         \begin{align*}
         \displaystyle \sup_{\partial B_r(y_i)\cap  \{u>0\}} u\geq C_2 r^{\frac{p}{p-1}},
         \end{align*}
         which and  the continuity of minimizers (Theorem~\ref{c1ar}) imply {the desired result}.
         \end{proof}

         {Recall that a} set $E$ is said to be porous with porosity {constant $\sigma\in(0,1)$, if there exists a constant  $R>0$} {such that for all $x \in E,$ and $r\in(0,R)$ there exists $y \in  \mathbb{R}^N$ such that  ${B_{\sigma  r}(y)}\subset B_r(x)\setminus E.$}

         It is well known that the {Hausdorff} dimension of a porous set does not exceed $N - C\sigma ^N$, where $C = C(N) > 0$ is a constant that
         depends only on $N$. Thus, the $N$-dimensional Lebesgue measure of the porous set is zero; see, {e.g., \cite{karp2000porosity,koskela1997hausdorff}}.

         The non-degeneracy and optimal growth of minimizers imply {the local porosity of the free boundary. More specifically, we state the following result, whose proof is standard (see \cite{karp2000porosity,challal2009porosity,hu2024porosity}) and hence is omitted}.
         \begin{proposition}\label{p33}
       Let $u$ be a
           minimizer of $\mathcal{J} (u)$ with  $\lambda_2>0$. Then, for every compact set $K\subset \Omega $, the intersection $K\cap \partial \{u>0\} $ is porous
         with porosity constant $\sigma$ depending only on $N$, $p$, $\lambda_1$, $\lambda_2$, $\| g\| _{{L^\infty}(\Omega )}$, {$\| g\| _{W^{1.p}(\Omega )}$, and} {the} diameter of $\Omega $.
         \end{proposition}

         \section{The corresponding penalized problem when $\lambda_2>0$}\label{ssec1}

In order to investigate  Hausdorff measure of the free boundary for a minimizer when $\lambda_2>0$, we consider the {following penalized problem} over the set $\mathcal{K}$:
         \begin{equation*}
          {\mathcal{J}_\varepsilon (u)}:=\int_{\Omega }  \left(\frac{1}{p}\left(\varepsilon +|\nabla u|^2\right)^\frac{p}{2}+ \lambda_1 \left(1-(u^+)^2\right)^2+ {\lambda_2\int_{0}^{u}    \vartheta  _\varepsilon (t)\text{d}t}\right)\text{d}x\rightarrow \text{min},
           \end{equation*}
  {where} for each $\varepsilon \in(0,1)$, $\vartheta _\varepsilon: \mathbb{R} \rightarrow [0,1]$ is {a} non-decreasing Lipschitz function given by
        \begin{equation*}
         {\vartheta _\varepsilon (t):=}\left\{\begin{aligned}
           &0, &&t<0,\\
           &\frac{t}{\varepsilon },&& 0<t\leq \varepsilon,\\
           &1,&&t>\varepsilon .
           \end{aligned}\right.
         \end{equation*}

{Note} that $|\nabla u|^p\leq(\varepsilon +|\nabla u|^2)^\frac{p}{2}$ and {$\left|\int_{0}^{u}    \vartheta  _\varepsilon (t)\text{d}t\right|\leq |u|$}. {Checking the proof of Theorem~\ref{t1}, it is easy to prove the existence and uniform boundedness for  minimizers of $\mathcal{J}_\varepsilon(u)$. In the sequel,  without special statements, we always denote the minimizer of $\mathcal{J}_\varepsilon(u)$ by $u_\varepsilon$.
Moreover,  $\|u_\varepsilon\|_{{L^\infty}(\Omega)}$ depends only on $N$, $p$, $\lambda_1$, $\lambda_2$, $\| g\| _{{L^\infty}(\Omega )}$, $\| g\| _{W^{1.p}(\Omega )}$, and  the diameter of $\Omega $,}


          Now, we {prove the following result.}

         \begin{lemma}\label{almoste}
         The  minimizer $u_\varepsilon $ of  $\mathcal{J}_\varepsilon(u)$ with   $\lambda_2>0$    {satisfies}
         \begin{equation}\label{hv2}
             \left\{\begin{aligned}
             & -\divi~\left(\left(\varepsilon +|\nabla u|^2\right)^\frac{p-2}{2}\nabla u\right)+4\lambda_1 \left(-u^++(u^+)^3\right)+\lambda_2\vartheta _\varepsilon (u)=0   &&\text{in}\ \Omega  \\
             & \  u =g  && \text{on}\ \partial \Omega
             \end{aligned}\right.
             \end{equation}
           {for almost every $x\in \Omega$.}
         \end{lemma}

         \begin{proof}
         Let {$u_\varepsilon$} be a minimizer of {$\mathcal{J}_\varepsilon(u)$}  with  $\lambda_2>0$. {It is clear that $u_\varepsilon $ is a weak solution to \eqref{hv2}}.
        {Moreover,    $u_\varepsilon$  is {locally $C^{1,\alpha }$}-continuous with some $\alpha\in( 0,1)$; see \cite{tolksdorf1984regularity}}. Furthermore, {for any $\Omega'\Subset \Omega$},    {$\|u_\varepsilon\|_{C^{1,\alpha }(\Omega')}$} depends only on  $N$, $p$, $\lambda_1$, $\lambda_2$, $\| g\| _{{L^\infty}(\Omega )}$, $\| g\| _{W^{1.p}(\Omega )}$, {$\Omega'$,} the diameter of $\Omega $,  and $\|u_\varepsilon\|_{L^\infty(\Omega)}$, while  $\|u_\varepsilon\|_{{L^\infty}(\Omega)}$
         depends only on $N$, $p$, $\lambda_1$, $\lambda_2$, $\| g\| _{{L^\infty}(\Omega )}$, $\| g\| _{W^{1.p}(\Omega )}$, and  the diameter of $\Omega $,
           Therefore, $\|u_\varepsilon\|_{C^{1,\alpha }(\Omega')}$ is independent of $\varepsilon$.
                 In particular,  there exists a positive constant   $c_1$  depending only on $N$, $p$, $\lambda_1$, $\lambda_2$, $\| g\| _{{L^\infty}(\Omega )}$, $\| g\| _{W^{1.p}(\Omega )}$, {$\Omega'$, and} the diameter of $\Omega $  such that
                  \begin{align}\label{uec1}
                       \|   u_\varepsilon\|_{L^\infty(\Omega')}+\| \nabla u_\varepsilon\|_{L^\infty(\Omega')} \leq c_1.
                  \end{align}

 {In view of   Lemma~\ref{f1l},  it} suffices to show that
             $\left(\left(\varepsilon +|\nabla u_\varepsilon|^2\right)^\frac{p-2}{2}\nabla u_\varepsilon \right)_{x_i}\in L^1_\text{loc}(\Omega )$.
         Indeed, it holds that
             \begin{align}
          \left(\left(\varepsilon +|\nabla u_\varepsilon|^2\right)^\frac{p-2}{2}\nabla u_\varepsilon \right)_{x_i}
         =& (p-2)\left(\varepsilon +|\nabla u_\varepsilon|^2\right)^\frac{p-4}{2}\nabla u_\varepsilon \nabla u_{\varepsilon x_i} \nabla u_\varepsilon +\left(\varepsilon +|\nabla u_\varepsilon|^2\right)^\frac{p-2}{2}\nabla u_{\varepsilon x_i} \notag\\
         \leq&  p\left(\varepsilon +|\nabla u_\varepsilon|^2\right)^\frac{p-4}{2}\nabla u_\varepsilon \nabla u_{\varepsilon x_i} \nabla u_\varepsilon  +\left(\varepsilon +|\nabla u_\varepsilon|^2\right)^\frac{p-2}{2}\nabla u_{\varepsilon x_i}  \notag\\
         \leq&  {p\left(\varepsilon +|\nabla u_\varepsilon|^2\right)^\frac{p-2}{2} \left| \nabla u_{\varepsilon x_i}\right|  +\left(\varepsilon +|\nabla u_\varepsilon|^2\right)^\frac{p-2}{2}\left|\nabla u_{\varepsilon x_i}\right| } \notag\\
         \leq&   (p+1)\left(\varepsilon +|\nabla u_\varepsilon|^2\right)^\frac{p-2}{2}\left|D^2 u_{\varepsilon }\right|.\label{A}
         \end{align}

         In the following, we consider two  {cases: $p\in[2,+\infty)$} and $p\in\left(\frac{2N}{N+2},2\right)$.

         {\textbf{Case 1: $p\in[2,+\infty)$.}} Define $\Psi :=u_{\varepsilon x_i}\phi  ^2$ {with} $\phi  \in C_0^\infty (B_{{3r}/{5}})$ {satisfying}
         \begin{equation*}
           \left\{\begin{aligned}
           & 0\leq \phi  \leq 1,&&\text{in}\  B_{{3r}/{5}},\\
           & \phi  =1,&& \text{in}\  B_{{r}/{2}},\\
           &| \nabla \phi  |\leq \frac{4}{r},&&\text{in}\  B_{{3r}/{5}},
           \end{aligned}\right.
           \end{equation*}
        {where $B_r\Subset \Omega $ is a ball with  some fixed radius $r>0$}.

        {Differentiating  \eqref{hv2} with respect to $x_i$, then, multiplying it by $\Psi$ and  integrating over} $B_{{3r}/{5}}$, we have
         \begin{align}\label{aes1}
           I_1:=&\int_{B_{{3r}/{5}}}   \left(\left(\varepsilon +|\nabla u_\varepsilon|^2\right)^\frac{p-2}{2}\nabla u_\varepsilon \right)_{x_i}\nabla \Psi  \text{d}x\notag\\
           =&-\int_{B_{{3r}/{5}}}  \left(4\lambda_1 \left(-u_\varepsilon^+ +\left(u_\varepsilon^+\right) ^3\right)+\lambda_2\vartheta _\varepsilon (u_\varepsilon )\right)_{x_i}\Psi  \text{d}x\notag\\
           =&-{4\lambda_1} \int_{B_{{3r}/{5}}}  \left(-u_\varepsilon^+ +\left(u_\varepsilon ^+\right)^3\right)_{x_i}{\Psi} \text{d}x -{\lambda_2}\int_{B_{{3r}/{5}}}  \left(\vartheta _\varepsilon (u_\varepsilon )\right)_{x_i}\Psi \text{d}x\notag\\
           =&:I_2+I_3.
         \end{align}

          By the definition of $\Psi$, $I_1$ becomes
         \begin{align*}
           I_1=&\int_{B_{{3r}/{5}}}   \left(\left(\varepsilon +|\nabla u_\varepsilon|^2\right)^\frac{p-2}{2}\nabla u_\varepsilon \right)_{x_i}\nabla \Psi  \text{d}x\notag\\
           =&\int_{B_{{3r}/{5}}}   \left((p-2)\left(\varepsilon +|\nabla u_\varepsilon|^2\right)^\frac{p-4}{2}\nabla u_\varepsilon \nabla u_{\varepsilon x_i} \nabla u_\varepsilon +\left(\varepsilon +|\nabla u_\varepsilon|^2\right)^\frac{p-2}{2}\nabla u_{\varepsilon x_i} \right)\nabla \Psi \text{d}x\notag\\
            =&   \int_{B_{{3r}/{5}}}       \left((p-2)\left(\varepsilon +|\nabla u_\varepsilon|^2\right)^\frac{p-4}{2}\nabla u_\varepsilon \nabla u_{\varepsilon x_i} \nabla u_\varepsilon +\left(\varepsilon +|\nabla u_\varepsilon|^2\right)^\frac{p-2}{2}\nabla u_{\varepsilon x_i}   \right)  \nabla u_{\varepsilon x_i}\phi  ^2 \text{d}x\notag\\
            &+\int_{B_{{3r}/{5}}}   2\left((p-2)\left(\varepsilon +|\nabla u_\varepsilon|^2\right)^\frac{p-4}{2}\nabla u_\varepsilon \nabla u_{\varepsilon x_i} \nabla u_\varepsilon +\left(\varepsilon +|\nabla u_\varepsilon|^2\right)^\frac{p-2}{2}\nabla u_{\varepsilon x_i} \right) u_{\varepsilon x_i}\phi  \nabla \phi  \text{d}x\notag\\
            =& \int_{B_{{3r}/{5}}}   \left((p-2)\left(\varepsilon +|\nabla u_\varepsilon|^2\right)^\frac{p-4}{2}|\nabla u_\varepsilon \nabla u_{\varepsilon x_i}|^2 +\left(\varepsilon +|\nabla u_\varepsilon|^2\right)^\frac{p-2}{2}\left|\nabla u_{\varepsilon x_i}\right|^2 \right)\phi  ^2 \text{d}x\notag\\
            &+  \int_{B_{{3r}/{5}}}       2\left((p-2)\left(\varepsilon +|\nabla u_\varepsilon|^2\right)^\frac{p-4}{2}  \nabla u_\varepsilon \nabla u_{\varepsilon x_i} \nabla u_\varepsilon   +  \left(\varepsilon +|\nabla u_\varepsilon|^2\right)^\frac{p-2}{2}\nabla u_{\varepsilon x_i} \right)   u_{\varepsilon x_i}\phi  \nabla \phi  \text{d}x\notag\\
             =&:I_{1}^1+I_{1}^2.
         \end{align*}

Since $p\in[2,+\infty)$, we get
         \begin{align}\label{l122a}
           I_1^1\geq& \int_{B_{{3r}/{5}}}   \left(\varepsilon +|\nabla u_\varepsilon|^2\right)^\frac{p-2}{2}\left|\nabla u_{\varepsilon x_i}\right|^2 \phi ^2  \text{d}x.
         \end{align}

In view of the definition of $\phi $, we deduce by the Young's inequality and {$\| \nabla u_\varepsilon\|_{L^\infty(B_r )} \leq c_1$} (see \eqref{uec1}) that
         \begin{align}\label{l222a}
            |I_1^2| \leq &   \int_{B_{{3r}/{5}}}     \bigg|2\left((p-2)\left(\varepsilon +|\nabla u_\varepsilon|^2\right)^\frac{p-4}{2}\nabla u_\varepsilon \nabla u_{\varepsilon x_i} \nabla u_\varepsilon   +  \left(\varepsilon +|\nabla u_\varepsilon|^2\right)^\frac{p-2}{2}  \nabla u_{\varepsilon x_i}   \right) u_{\varepsilon x_i}\phi  \nabla \phi \bigg| \text{d}x\notag\\
           \leq&\int_{B_{{3r}/{5}}}   2(p-1)\left(\varepsilon +|\nabla u_\varepsilon|^2\right)^\frac{p-2}{2} \left|\nabla u_{\varepsilon x_i}\right| \cdot  \left| u_{\varepsilon x_i}\right| \cdot  | \nabla \phi |\cdot \phi \text{d}x\notag\\
           \leq& \frac{8}{r}\int_{B_{{3r}/{5}}}  (p-1)\left(\varepsilon +|\nabla u_\varepsilon|^2\right)^\frac{p-2}{2}\left|\nabla u_{\varepsilon x_i}\right|\cdot \left| u_{\varepsilon x_i}\right| \phi  \text{d}x\notag\\
           \leq &\frac{1}{2}\int_{B_{{3r}/{5}}}     \left(\varepsilon +|\nabla u_\varepsilon|^2\right)^\frac{p-2}{2} \left|\nabla u_{\varepsilon x_i}\right|^2  \phi  ^2\text{d}x+\frac{32(p-1)^2}{r^2}\int_{B_{{3r}/{5}}}  \left(\varepsilon +|\nabla u_\varepsilon|^2\right)^\frac{p-2}{2} \left|u_{\varepsilon x_i}\right|^2 \text{d}x\notag\\
           \leq &{\frac{1}{2}\int_{B_{{3r}/{5}}}     \left(\varepsilon +|\nabla u_\varepsilon|^2\right)^\frac{p-2}{2} \left|\nabla u_{\varepsilon x_i}\right|^2  \phi  ^2\text{d}x+\frac{32(p-1)^2}{r^2}\int_{B_{{3r}/{5}}}  \left(\varepsilon +|\nabla u_\varepsilon|^2\right)^\frac{p-2}{2} \left|\nabla u_{\varepsilon}\right|^2 \text{d}x}\notag\\
           \leq &\frac{1}{2}\int_{B_{{3r}/{5}}}   \left(\varepsilon +|\nabla u_\varepsilon|^2\right)^\frac{p-2}{2} \left|\nabla u_{\varepsilon x_i}\right|^2  \phi  ^2\text{d}x+\frac{32(p-1)^2}{r^2} \left(1+c_1^2\right)^\frac{p-2}{2}c^2_1 C(N)r^N \notag\\
           \leq &{\frac{1}{2}\int_{B_{{3r}/{5}}}   \left(\varepsilon +|\nabla u_\varepsilon|^2\right)^\frac{p-2}{2} \left|\nabla u_{\varepsilon x_i}\right|^2  \phi  ^2\text{d}x+C(c_1,N,p)r^{N-2}}.
         \end{align}

{For $I_2$, we deduce that}
          \begin{align}\label{r122a}
           {I_2=}&-{4\lambda_1}\int_{B_{{3r}/{5}}}   \left(-u_\varepsilon^+ +\left(u_\varepsilon^+\right)^3\right)_{x_i}\Psi\text{d}x \notag\\
           =& -{4\lambda_1}\int_{B_{{3r}/{5}}}   \left(- \chi_{\{u_{\varepsilon}>0\}} +3\left(u_\varepsilon^+\right)^2 \right)u_{\varepsilon x_i}^2 \phi ^2\text{d}x \notag\\
           \leq &{4\lambda_1}\int_{B_{{3r}/{5}}}  u_{\varepsilon x_i}^2\phi ^2\text{d}x \notag\\
           \leq &{4\lambda_1}\int_{B_{{3r}/{5}}}   |\nabla u_{\varepsilon }| ^2\text{d}x\notag\\
           \leq &  {C(c_1,N,\lambda _1)r^N},
         \end{align}
         {where $\chi _{\{\cdot\}}$ is the standard characteristic function.}

 Since $\vartheta_\varepsilon  ({t})$ is non-decreasing in $t$ and    $\lambda_2>0$, it holds that
         \begin{align}\label{r222a}
           I_3 = -{\lambda_2}\int_{B_{{3r}/{5}}}  \left(\vartheta _\varepsilon (u_\varepsilon )\right)_{x_i}\Psi\text{d}x
               = -{\lambda_2}\int_{B_{{3r}/{5}}} \vartheta '_\varepsilon (u_\varepsilon )u_{\varepsilon x_i}^2\phi  ^2\text{d}x
               \leq  0.
         \end{align}

 By virtue of \eqref{aes1}, \eqref{l122a}, \eqref{l222a}, {\eqref{r122a}, and} \eqref{r222a}, we have
       {\begin{align}\label{lr22a}
          \int_{B_{{r}/{2}}}  \left(\varepsilon +|\nabla u_\varepsilon|^2\right)^\frac{p-2}{2} |\nabla u_{\varepsilon x_i}|^2  \text{d}x
          \leq &\int_{B_{{3r}/{5}}} \left(\varepsilon +|\nabla u_\varepsilon|^2\right)^\frac{p-2}{2} |\nabla u_{\varepsilon x_i}|^2 \phi^2 \text{d}x\notag\\
          \leq &C(c_1,N,\lambda _1)r^N+C(c_1,N,p)r^{N-2}\notag\\
          \leq &C(c_1,N,p,\lambda _1,\Omega )r^{N-2}.
         \end{align}}

         {Summing up \eqref{lr22a} from $i=1$ to $N$, it follows that
         \begin{align}\label{2jdgj}
         \int_{B_{{r}/{2}}}   \left(\varepsilon +|\nabla u_\varepsilon|^2\right)^\frac{p-2}{2} |D^2 u_{\varepsilon}|^2  \text{d}x
          \leq C_3 r^{N-2},
         \end{align}}
         where $C_3$ depends only on $c_1$, $N$, $p$, {$\lambda_1$, and} the diameter of $\Omega $.

         {\textbf{Case 2: $p\in\left(\frac{2N}{N+2},2\right)$}.} Let $\Psi =\left(\varepsilon +{u^2_{\varepsilon x_i}}\right)^\frac{p-2}{2}u_{\varepsilon x_i}\phi  ^2$ {with} $\phi  \in C_0^\infty (B_{{3r}/{5}})$ {satisfying}
         \begin{equation*}
         \left\{\begin{aligned}
         & 0\leq \phi  \leq 1,&& \text{in}\  B_{{3r}/{5}},\\
         & \phi  =1,&&\text{in}\ B_{{r}/{2}},\\
         &| \nabla \phi  |\leq \frac{4}{r},&& \text{in}\  B_{{3r}/{5}},
         \end{aligned}\right.
         \end{equation*}
      {where $B_r\Subset \Omega $ is a ball with  some fixed radius $r>0$}.	

      Differentiating \eqref{hv2} with respect to $x_i$, then, multiplying it by $\Psi$ and  integrating over $B_{{3r}/{5}}$, we have
         \begin{align}\label{almost}
           I_1:=&\int_{B_{{3r}/{5}}}   \left(\left(\varepsilon +|\nabla u_\varepsilon|^2\right)^\frac{p-2}{2}\nabla u_\varepsilon \right)_{x_i}\nabla \Psi  \text{d}x\notag\\
           =&-\int_{B_{{3r}/{5}}}  \left(4\lambda_1 \left(-u_\varepsilon^+ +\left(u_\varepsilon^+\right) ^3\right)+\lambda_2\vartheta _\varepsilon (u_\varepsilon )\right)_{x_i}\Psi  \text{d}x\notag\\
           =&-{4\lambda_1} \int_{B_{{3r}/{5}}}  \left(-u_\varepsilon^+ +\left(u_\varepsilon ^+\right)^3\right)_{x_i}{\Psi} \text{d}x -{\lambda_2}\int_{B_{{3r}/{5}}}  \left(\vartheta _\varepsilon (u_\varepsilon )\right)_{x_i}\Psi \text{d}x\notag\\
           =&:I_2+I_3.
         \end{align}

  The left-hand side of \eqref{almost} becomes
         \begin{align*}
         I_1=&\int_{B_{{3r}/{5}}}   \left(\left(\varepsilon +|\nabla u_\varepsilon| ^2\right)^\frac{p-2}{2}\nabla u_\varepsilon \right)_{x_i}\nabla \Psi\text{d}x\notag\\
         =&\int_{B_{{3r}/{5}}}   \left((p-2)\left(\varepsilon +|\nabla u_\varepsilon| ^2\right)^\frac{p-4}{2}\nabla u_\varepsilon\nabla u_{\varepsilon x_i}\nabla u_{\varepsilon} +\left(\varepsilon +|\nabla u_\varepsilon| ^2\right)^\frac{p-2}{2} \nabla u_{\varepsilon x_i}\right)\nabla\Psi\text{d}x\notag\\
         =&\int_{B_{{3r}/{5}}}   \left((p-2)\left(\varepsilon +|\nabla u_\varepsilon| ^2\right)^\frac{p-4}{2}\left|\nabla u_\varepsilon \nabla u_{\varepsilon x_i}\right|^2 +\left(\varepsilon +|\nabla u_\varepsilon| ^2\right)^\frac{p-2}{2} \left|\nabla u_{\varepsilon x_i}\right|^2\right)\notag\\
          &\qquad\times \left((p-2)\left(\varepsilon +u_{\varepsilon x_i}^2 \right)^\frac{p-4}{2}u_{\varepsilon x_i}^2+\left(\varepsilon +u_{\varepsilon x_i}^2 \right)^\frac{p-2}{2}\right)\phi  ^2 \text{d}x\notag\\
         &+\int_{B_{{3r}/{5}}}   2\left((p-2)\left(\varepsilon +|\nabla u_\varepsilon| ^2\right)^\frac{p-4}{2}\nabla u_\varepsilon\nabla u_{\varepsilon x_i}\nabla u_{\varepsilon} +\left(\varepsilon +|\nabla u_\varepsilon| ^2\right)^\frac{p-2}{2}\nabla u_{\varepsilon x_i} \right) \left(\varepsilon +u_{\varepsilon x_i}^2\right)^\frac{p-2}{2}u_{\varepsilon x_i}\phi  \nabla \phi  \text{d}x\notag\\
         =:&I_1^1+I_1^2.
         \end{align*}

  Since $p\in\left(\frac{2N}{N+2},2\right)$, we get
         \begin{align}\label{aes2}
         I_1^1
         \geq&(p-1)^2\int_{B_{{3r}/{5}}} \left(\varepsilon +|\nabla u_\varepsilon| ^2\right)^\frac{p-2}{2}\left(\varepsilon + u_{\varepsilon x_i}^2\right)^\frac{p-2}{2}\left|\nabla u_{\varepsilon x_i}\right|^2\phi  ^2\text{d}x.
         \end{align}

 Using the fact that $|p-2|<p$, the Young's inequality, and {$\| \nabla u_\varepsilon\|_{L^\infty(B_r )} \leq c_1$} (see \eqref{uec1}), we have
         \begin{align}\label{l2224a}
         |I_1^2|
         \leq&\int_{B_{{3r}/{5}}}   \bigg|2 \left((p-2)\left(\varepsilon +|\nabla u_\varepsilon| ^2\right)^\frac{p-4}{2}\nabla u_\varepsilon \nabla u_{\varepsilon x_i}\nabla u_{\varepsilon }+\left(\varepsilon +|\nabla u_\varepsilon| ^2\right)^\frac{p-2}{2}\nabla u_{\varepsilon x_i}\right) \left(\varepsilon +u_{\varepsilon x_i}^2\right)^\frac{p-2}{2} u_{\varepsilon x_i}  \phi  \nabla \phi  \bigg|  \text{d}x\notag\\
         \leq &\frac{8}{r}\int_{B_{{3r}/{5}}}    (p+1)\left(\varepsilon +|\nabla u_\varepsilon| ^2\right)^\frac{p-2}{2} \left|\nabla u_{\varepsilon x_i}\right|\cdot \left(\varepsilon +u_{\varepsilon x_i}^2\right)^\frac{p-2}{2} \left|u_{\varepsilon x_i} \right| \phi    \text{d}x\notag\\
         \leq &\frac{(p-1)^2}{2}\int_{B_{{3r}/{5}}}   \left(\varepsilon +|\nabla u_\varepsilon| ^2\right)^\frac{p-2}{2}  \left|\nabla u_{\varepsilon x_i}\right|^2  \left(\varepsilon +u_{\varepsilon x_i}^2\right)^\frac{p-2}{2}\phi  ^2\text{d}x  \notag\\
         &+\frac{32(p+1)^2}{r^2(p-1)^2}\int_{B_{{3r}/{5}}}   \left(\varepsilon +\nabla u_{\varepsilon }^2\right)^\frac{p-2}{2}  \left(\varepsilon +u_{\varepsilon x_i}^2\right)^\frac{p-2}{2}\left|u_{\varepsilon x_i}\right|^2 \text{d}x\notag\\
         \leq &\frac{(p-1)^2}{2}\int_{B_{{3r}/{5}}}   \left(\varepsilon +|\nabla u_\varepsilon| ^2\right)^\frac{p-2}{2}  \left|\nabla u_{\varepsilon x_i}\right|^2  \left(\varepsilon +u_{\varepsilon x_i}^2\right)^\frac{p-2}{2}\phi  ^2\text{d}x+\frac{32(p+1)^2}{r^2(p-1)^2}\int_{B_{{3r}/{5}}}   |\nabla u_{\varepsilon }|^{p-2} \left|u_{\varepsilon x_i}\right|^p \text{d}x\notag\\
         \leq &\frac{(p-1)^2}{2}\int_{B_{{3r}/{5}}}   \left(\varepsilon +|\nabla u_\varepsilon| ^2\right)^\frac{p-2}{2}|\nabla u_{\varepsilon x_i}|^2  \left(\varepsilon +u_{\varepsilon x_i}^2\right)^\frac{p-2}{2}\phi  ^2\text{d}x+\frac{32(p+1)^2}{r^2(p-1)^2}\int_{B_{{3r}/{5}}}   |\nabla u_\varepsilon |^{2(p-1)}  \text{d}x\notag\\
        \leq & \frac{(p-1)^2}{2} \int_{B_{{3r}/{5}}}  \left(\varepsilon +|\nabla u_\varepsilon| ^2\right)^\frac{p-2}{2} |\nabla u_{\varepsilon x_i}|^2  \left(\varepsilon +u_{\varepsilon x_i}^2\right)^\frac{p-2}{2}\phi  ^2\text{d}x
        +C(c_1,N,p)r^{N-2} .
         \end{align}

 By the definition of $\phi  $, $I_2$ becomes
         \begin{align}\label{r1224a}
         I_2=&-{4\lambda_1}\int_{B_{{3r}/{5}}}  \left(-u_\varepsilon^+ +\left(u_\varepsilon^+\right)^3\right)_{x_i}\Psi \text{d}x \notag\\
         =&-{4\lambda_1}\int_{B_{{3r}/{5}}}  {\left(- \chi_{\{u_{\varepsilon}>0\}} +3\left(u_\varepsilon^+\right)^2 \right) \left(\varepsilon +u_{\varepsilon x_i}^2\right)^\frac{p-2}{2}u_{\varepsilon x_i}^2}\phi ^2\text{d}x \notag\\
         \leq&  {{4\lambda_1}\int_{B_{{3r}/{5}}}     \left(\varepsilon +u_{\varepsilon x_i}^2\right)^\frac{p-2}{2} u_{\varepsilon x_i}  ^2\phi ^2\text{d}x }\notag\\
         \leq &{4\lambda_1}\int_{B_{{3r}/{5}}}   \left|u_{\varepsilon x_i}\right|^p\phi ^2\text{d}x \notag\\
         \leq &{4\lambda_1}\int_{B_{{3r}/{5}}} |\nabla u_{\varepsilon }| ^p\text{d}x\notag\\
         \leq & {C(c_1,N,p,\lambda _1)r^N}.
         \end{align}

 Since $\vartheta_\varepsilon  ({t})$ is non-decreasing in $t$, it follows that
         \begin{align}\label{r2224a}
         I_3= &-{\lambda_2}\int_{B_{{3r}/{5}}}  \left( \vartheta _\varepsilon (u_\varepsilon )\right)_{x_i}\Psi  \text{d}x
         =  -{\lambda_2}\int_{B_{{3r}/{5}}}  \vartheta '_\varepsilon (u_\varepsilon )\left(\varepsilon + u_{\varepsilon x_i}^2\right)^\frac{p-2}{2}u_{\varepsilon x_i}^2\phi  ^2\text{d}x
         \leq    0.
         \end{align}

  By virtue of \eqref{almost}, \eqref{aes2}, \eqref{l2224a}, {\eqref{r1224a}, and} \eqref{r2224a}, it follows that
         \begin{align}\label{p12a}
          \int_{B_{{r}/{2}}}   \left(\varepsilon +\left|\nabla u_{\varepsilon }\right|^2\right)^{p-2}\left|\nabla u_{\varepsilon x_i}\right|^2 \text{d}x
              \leq&\int_{B_{{3r}/{5}}}   \left(\varepsilon +|\nabla u_{\varepsilon }|^2\right)^\frac{p-2}{2}|\nabla u_{\varepsilon x_i}|^2  \left(\varepsilon +u_{\varepsilon x_i}^2\right)^\frac{p-2}{2}\phi  ^2\text{d}x\notag\\
         \leq& C(c_1,N,p)r^{N-2}+C(c_1,N,p,\lambda _1)r^N\notag\\
         \leq& C(c_1,N,p,\lambda _1,\Omega )r^{N-2}.
         \end{align}

 Summing up \eqref{p12a} from $i=1$ to $N$, we have
         \begin{align}\label{lr224a}
         \int_{B_{{r}/{2}}}   \left(\left(\varepsilon +{|\nabla u_\varepsilon|^2} \right)^\frac{p-2}{2}|D^2 u_{\varepsilon}|\right)^2  \text{d}x
             \leq   C_4r^{N-2},
           \end{align}
         where $C_4$ depends only on $c_1$, $N$, $p$, {$\lambda_1$, and} the diameter of $\Omega $.

         {Finally}, {we infer from \eqref{2jdgj} and \eqref{lr224a} that} $\left(\varepsilon +|\nabla u_\varepsilon| ^2\right)^\frac{p-2}{2}|D^2u_\varepsilon |\in L^2_\text{loc}(\Omega )$,
         {which along with \eqref{A}, implies}
        \begin{align*}\left(\left(\varepsilon +|\nabla u_\varepsilon| ^2\right)^\frac{p-2}{2}\nabla u_\varepsilon \right)_{x_i}\in L^1_\text{loc}(\Omega ). \end{align*}
         \end{proof}

         {The following lemma  indicates the non-degeneracy of second-order term.}

         \begin{lemma}\label{erjie}
        The  minimizer $u_\varepsilon $ of {$\mathcal{J}_\varepsilon(u)$} with   $\lambda_2>0$  {satisfies}
         \begin{align*}
             \frac{\left(4\lambda_1 \left(-u_\varepsilon^++\left(u_\varepsilon ^+\right)^3\right)+\lambda_2\vartheta _\varepsilon (u_\varepsilon )\right)^2}{(p+1)^2}\leq \left(\left(\varepsilon +|\nabla u_\varepsilon| ^2\right)^\frac{p-2}{2} |D^2 u_\varepsilon |\right)^2 \ {\text{a.e.\ in}}\ \Omega.
         \end{align*}
         \end{lemma}
         \begin{proof}
         Let $\delta  _{ij}=1$ if $i=j$ and $\delta _{ij}=0$ if $i\neq j$. Indeed, by direct computations, we have
         \begin{align*}\label{erjie22}
           \left| 4\lambda_1 \left(-u_\varepsilon^+ +\left(u_\varepsilon^+\right) ^3\right)+\lambda_2\vartheta _\varepsilon (u_\varepsilon )\right| ^2
         =& \left(\sum_{i,j = 1}^{N} \left(\left(\varepsilon +|\nabla u_\varepsilon|^2\right)^\frac{p-2}{2}\delta _{ij}u_{\varepsilon x_ix_j}+(p-2)\left(\varepsilon +|\nabla u_\varepsilon|^2\right)^\frac{p-4}{2}u_{\varepsilon x_i} u_{\varepsilon x_j}u_{\varepsilon x_ix_j}\right) \right)^2\notag\\
           \leq&\left(\sum_{i,j = 1}^{N} \left(\left(\varepsilon +|\nabla u_\varepsilon|^2\right)^\frac{p-2}{2}\delta   _{ij}\left|u_{\varepsilon x_ix_j}\right|+p\left(\varepsilon +|\nabla u_\varepsilon|^2\right)^\frac{p-4}{2}\left|u_{\varepsilon x_i} u_{\varepsilon x_j}u_{\varepsilon x_ix_j}\right|\right) \right)^2\notag\\
           \leq &(p+1)^2\left(\left(\varepsilon +|\nabla u_\varepsilon|^2\right)^\frac{p-2}{2} |D^2 u_\varepsilon |\right)^2.
           \end{align*}
         \end{proof}

         For the minimizer $u_\varepsilon $ of {$\mathcal{J}_\varepsilon(u)$}  with  $\lambda_2>0$,  since $u_\varepsilon $ is uniformly bounded in {$C^{1,\alpha  }(\Omega')$   for any $\Omega' \Subset\Omega$ and  some} $\alpha \in(0,1)$, there exists a function $u $ such that $u_\varepsilon \rightarrow u$ in {$C^{1,\beta  }(\Omega'')$ for all $\beta\in (0,\alpha)$ and $\Omega''\Subset \Omega'$. It is easy to show that such a function $u$ is a non-negative minimizer of \eqref{ju}, namely,  the following proposition holds true. The proof is omitted.}
         \begin{proposition}\label{equal}
         {The limit of $u_\varepsilon $, denoted by $u$, is a non-negative minimizer of $\mathcal{J} (u)$  with  $\lambda_2>0$.}
         \end{proposition}

         \section{Hausdorff measure of the free boundary when $\lambda_2>0$}\label{sec5}

         In this section, {we  show  that there exists {at least one}  minimizer with  finite {($N-1$)}{-dimensional Hausdorff} measure for the corresponding free boundary  when  $\lambda_2>0$.}

        First, {for $\sigma \in(0,1)$ and any minimizer, {denoted by $u$}, of $\mathcal{J}(u)$,} we define
            $ O_\sigma  :=\left\{|\nabla u |\leq \sigma  ^\frac{1}{p-1}\right\}$ and $  O_{\sigma   i}:=\left\{| u_{x_i} |\leq \sigma  ^\frac{1}{p-1}\right\}.$
          We prove the following lemma.
               \begin{lemma}\label{LL22} Assume that    $\lambda_2>0$.
                 Let $r_1$ and $C_0$ be positive constants {determined by} {Proposition~\ref{the32} and Theorem~\ref{the31}, respectively}. Let {$B_{r_2}\Subset \Omega $ be a ball with sufficiently small radius $r_2$ satisfying}
         $r_2\leq\min \left\{r_1,\left(\frac{1}{3 C_0}\right)^{\frac{p-1}{p}}\right\}$.
          {Then, there} exists at least one  minimizer, {denoted by $u$}, of $\mathcal{J}(u)$  such that for every {$x_0\in \Gamma ^+\cap B_{r_2}$, any $\sigma\in(0,1)$, and {any  $r\in (0,{r_2})$} with  $B_r(x_0)\subset B_{r_2}$} there holds
         \begin{align*}
         \int_{0}^{1}  \mathcal{L}^N\left({O_\sigma}  \cap B_{rs}(x_0)\cap \{u>0\}\right){\text{d}s}\leq  C \sigma r^{N-1},
         \end{align*}
               where  $C$ is a {positive} constant  depending only on $r_2$, $N$, $p$, $\lambda_1$, $\lambda_2$, $\| g\| _{{L^\infty}(\Omega )}$, {$\| g\| _{W^{1.p}(\Omega )}$,
                 and} {the} diameter of $\Omega $.
               \end{lemma}
            \begin{proof}
               Let  $u_\varepsilon  $  be a minimizer of {$J_\varepsilon (u)$, and $u$ be the limit of $u_\varepsilon  $ determined by Proposition~\ref{equal}.} Define
           \begin{align*}
             O_\varepsilon:=\left\{\left(\varepsilon +|\nabla u_\varepsilon| ^2\right)^\frac{1}{2}\leq 2\sigma  ^\frac{1}{p-1}\right\}\quad \text{and} \quad O_{\varepsilon i}:=\left\{\left(\varepsilon +\left| u_{\varepsilon x_i}\right| ^2\right)^\frac{1}{2}\leq 2\sigma  ^\frac{1}{p-1}\right\}.
             \end{align*}

           We claim that
         $
           \left(O_\sigma  \cap B_{r_2}\right) \subset \left(O_\varepsilon  \cap B_{r_2}\right).
        $
         Indeed, there exists $\varepsilon  _1$ such that {$\left\|\left(\varepsilon +|\nabla u_\varepsilon| ^2\right)^\frac{1}{2}-\nabla u\right\| _{L^\infty({\overline {B}_{r_2})}}\leq \sigma  ^\frac{1}{p-1}$ holds true for any $\varepsilon  \in(0,\varepsilon  _1)$.}
             Then, for any $x\in O_\sigma  \cap B_{r_2}$,  {it holds that}
         \begin{align*}
         \left\| \left(\varepsilon +|\nabla u_\varepsilon| ^2\right)^\frac{1}{2} \right\|_{L^\infty\left({\overline {B}_{r_2}}\right)} \leq & \left\| \left(\varepsilon +|\nabla u_\varepsilon| ^2\right)^\frac{1}{2}-\nabla u\right\| _{L^\infty\left({\overline {B}_{r_2}}\right)}+\big\| \nabla u \big\| _{{{L^\infty\left({\overline {B}_{r_2}}\right)}}}
         \leq  \sigma  ^\frac{1}{p-1}+\sigma  ^\frac{1}{p-1}
         = 2\sigma  ^\frac{1}{p-1},
           \end{align*}
which implies that such $x\in O_\varepsilon  \cap B_{r_2}$.

             We prove in two cases: $p\in[2,+\infty)$ and $p\in\left(\frac{2N}{N+2},2\right)$.

             {\textbf{Case 1: $p\in[2,+\infty)$.}} Let $H$ {be given} by
             \begin{equation*}
              H(\eta ){:=}\left\{\begin{aligned}
              & 2^{p-1}\sigma  ,&& \eta > 2\sigma  ^\frac{1}{p-1},\\
              & 2^{p-2}\sigma ^ \frac{p-2}{p-1}\eta  ,&& |\eta| \leq 2\sigma  ^\frac{1}{p-1},\\
              &-2^{p-1}\sigma ,&& \eta <- 2\sigma  ^\frac{1}{p-1}.\\
              \end{aligned}\right.
              \end{equation*}

 {In view of Lemma~\ref{almoste}, differentiating} \eqref{hv2} with respect to $x_i$ gives
             \begin{align}\label{dxi22}
                 0=&-\divi~\left((p-2)\left(\varepsilon +|\nabla u_\varepsilon| ^2\right)^\frac{p-4}{2}\nabla u_\varepsilon\nabla u_{\varepsilon x_i}\nabla u_{\varepsilon }+\left(\varepsilon +|\nabla u_\varepsilon| ^2\right)^\frac{p-2}{2} \nabla u_{\varepsilon x_i}\right)\notag\\
                 &+4\lambda_1{\left(- \chi_{\{u_{\varepsilon}>0\}} +3\left(u_\varepsilon^+\right)^2 \right)u_{\varepsilon x_i}}+\lambda_2\vartheta' _\varepsilon (u_\varepsilon )u_{\varepsilon x_i}  .
             \end{align}

 Multiplying \eqref{dxi22} by $H\left(u_{\varepsilon  x_i}\right)$ and integrating over $B_{rs}(x_0)$ {with $r\in(0,r_2)$}, we have
             \begin{align}\label{lebes1}
                &  \int_{B_{rs}(x_0)}     \left((p-2)\left(\varepsilon +|\nabla u_\varepsilon| ^2\right)^\frac{p-4}{2}  \nabla u_\varepsilon\nabla u_{\varepsilon x_i}\nabla u_{\varepsilon }  +  \left(\varepsilon +|\nabla u_\varepsilon| ^2\right)^\frac{p-2}{2}   \nabla u_{\varepsilon x_i}\right)  \nabla H\left(u_{\varepsilon  x_i}\right)\text{d}x\notag\\
                &  +  {4\lambda_1}\int_{B_{rs}(x_0)}    {\left(- \chi_{\{u_{\varepsilon}>0\}}   +3\left(u_\varepsilon^+\right)^2 \right)u_{\varepsilon x_i}}H\left(u_{\varepsilon  x_i}\right)\text{d}x  +  {\lambda_2}\int_{B_{rs}(x_0)}     \vartheta' _\varepsilon (u_\varepsilon )u_{\varepsilon x_i}H\left(u_{\varepsilon  x_i}\right)\text{d}x\notag\\
                =\!&  \int_{\partial B_{rs}(x_0)}     \left((p-2)\left(\varepsilon +|\nabla u_\varepsilon| ^2\right)^\frac{p-4}{2}  \nabla u_\varepsilon\nabla u_{\varepsilon x_i}\nabla u_{\varepsilon }  +  \left(\varepsilon +|\nabla u_\varepsilon|^2\right)^\frac{p-2}{2}   \nabla u_{\varepsilon x_i}\!\right)  H\left(u_{\varepsilon  x_i}\right){\boldsymbol{\nu}} \text{d}S,
             \end{align}
             where {$\boldsymbol{\nu} $} is the unit outward normal vector.

 {Using the H\"{o}lder's inequality, we deduce by \eqref{2jdgj} and $\| \nabla u_\varepsilon\|_{L^\infty(B_{r_2})} \leq c_1$ (see \eqref{uec1}) that}
             \begin{align}\label{3122}
              &\int_{0}^{1}  \int_{\partial B_{rs}(x_0)}  \left((p-2)\left(\varepsilon +|\nabla u_\varepsilon| ^2\right)^\frac{p-4}{2}\nabla u_\varepsilon\nabla u_{\varepsilon x_i}\nabla u_{\varepsilon }+\left(\varepsilon +|\nabla u_\varepsilon| ^2\right)^\frac{p-2}{2} \nabla u_{\varepsilon x_i}\right) H\left(u_{\varepsilon  x_i}\right){\boldsymbol{\nu}} \text{d}S  \text{d}s \notag\\
              \leq & (p-1)\int_{B_{r}(x_0)}  \left(\varepsilon +|\nabla u_\varepsilon| ^2\right)^\frac{p-2}{2} |D^2u_\varepsilon |\cdot |H\left(u_{\varepsilon  x_i}\right)|\text{d}x \notag\\
              \leq & (p-1) \int_{B_{r}(x_0)}  \left(\varepsilon +|\nabla u_\varepsilon| ^2\right)^\frac{p-2}{2} |D^2u_\varepsilon | 2^{p-1}\sigma  \text{d}x \notag\\
              \leq & 2^{p-1}(p-1)\sigma   \left(\int_{B_{r}(x_0)}  \left(\varepsilon +|\nabla u_\varepsilon| ^2\right)^\frac{p-2}{2} \left|D^2u_\varepsilon \right|^2 \text{d}x\right)^{\frac{1}{2}} \left(\int_{B_{r}(x_0)}  \left(\varepsilon +|\nabla u_\varepsilon| \right)^\frac{p-2}{2}\text{d}x\right)^{\frac{1}{2}}\notag\\
              \leq & 2^{p-1}(p-1)\sigma   \left(\int_{B_{r}(x_0)}  \left(\varepsilon +|\nabla u_\varepsilon| ^2\right)^\frac{p-2}{2} \left|D^2u_\varepsilon \right|^2 \text{d}x\right)^{\frac{1}{2}} \left(\int_{B_{r}(x_0)}  \left(1 +c_1 ^2\right)^\frac{p-2}{2}\text{d}x\right)^{\frac{1}{2}}\notag\\
              \leq & 2^{p-1}(p-1)\sigma   C_3^{\frac{1}{2}} r^{\frac{N-2}{2}}C(N,c_1,p)r^\frac{N}{2}\notag\\
              \leq & {C(c_1,C_3,N,p) \sigma  r^{N-1}}.
              \end{align}

By $\left|u_{\varepsilon x_i} \right|<\left(\varepsilon +\left|u_{\varepsilon x_i} \right|^2\right)^\frac{1}{2}$, we have $ \left\{(\varepsilon +\left| u_{\varepsilon x_i}\right| ^2)^\frac{1}{2}\leq 2\sigma  ^\frac{1}{p-1}\right\}\subset \left\{\left| u_{\varepsilon x_i}\right|\leq 2\sigma  ^\frac{1}{p-1}\right\},$ which, along with $O_\varepsilon \subset O_{\varepsilon i}$, implies
              \begin{align}\label{3222}
                  &\sum_{i = 1}^{N}  \int_{B_{rs}(x_0)}   \left((p-2)\left(\varepsilon +|\nabla u_\varepsilon| ^2\right)^\frac{p-4}{2}\nabla u_\varepsilon\nabla u_{\varepsilon x_i}  \nabla u_{\varepsilon }   +\left(\varepsilon +|\nabla u_\varepsilon| ^2\right)^\frac{p-2}{2} \nabla u_{\varepsilon x_i}\right)\notag\\
                  &\qquad \cdot\nabla H\left(u_{\varepsilon  x_i}\right)\text{d}x\notag\\
                 \geq &\sum_{i = 1}^{N}  \int_{B_{rs}(x_0)\cap O_{\varepsilon i}}   \left((p-2)\left(\varepsilon +|\nabla u_\varepsilon| ^2\right)^\frac{p-4}{2}\nabla u_\varepsilon\nabla u_{\varepsilon x_i}\nabla u_{\varepsilon }+\left(\varepsilon +|\nabla u_\varepsilon| ^2\right)^\frac{p-2}{2} \nabla u_{\varepsilon x_i}\right)\notag\\
                 &\qquad\cdot 2^{p-2}\sigma ^ \frac{p-2}{p-1}\nabla u_{\varepsilon x_i}\text{d}x\notag\\
                  \geq &\sum_{i = 1}^{N}  \int_{B_{rs}(x_0)\cap O_{\varepsilon }}   \left(\varepsilon +|\nabla u_\varepsilon| ^2\right)^\frac{p-2}{2}\left|\nabla u_{\varepsilon x_i}\right|^2 \left(\varepsilon +|\nabla u_\varepsilon| ^2\right)^\frac{p-2}{2}\text{d}x\notag\\
                  = &\int_{B_{rs}(x_0)\cap O_{\varepsilon }}   \left(\left(\varepsilon +|\nabla u_\varepsilon| ^2\right)^\frac{p-2}{2}\left|D^2 u_{\varepsilon }\right|\right)^2 \text{d}x\notag\\
                  \geq &\int_{B_{rs}(x_0)\cap O_{\sigma   }}   \left(\left(\varepsilon +|\nabla u_\varepsilon| ^2\right)^\frac{p-2}{2}\left|D^2 u_{\varepsilon }\right| \right)^2\text{d}x.
              \end{align}

Since $\vartheta _\varepsilon (t)$ is non-decreasing in $t$ and $H(\eta )\eta \geq 0$, it follows that
              \begin{align}\label{3322}
                  \int_{B_{rs}(x_0)}  \lambda_2\vartheta' _\varepsilon (u_\varepsilon )u_{\varepsilon x_i}H\left(u_{\varepsilon  x_i}\right)\text{d}x\geq 0.
              \end{align}

By the definition of $H$ and $\left(\|  u_\varepsilon\|_{L^\infty(B_{r_2})}+\| \nabla u_\varepsilon\|_{L^\infty(B_{r_2})}\right) \leq c_1$ (see \eqref{uec1}), it follows that
              \begin{align}\label{3422}
                   \left|{4\lambda_1}\int_{B_{rs}(x_0)}   \left(- \chi_{\{u_{\varepsilon}>0\}} +3\left(u_\varepsilon^+\right)^2 \right)u_{\varepsilon x_i} H\left(u_{\varepsilon  x_i}\right)\text{d}x\right|
                  \leq &{4\lambda_1}\int_{B_{rs}(x_0)}  \left| \left(- \chi_{\{u_{\varepsilon}>0\}} +3\left(u_\varepsilon^+\right)^2 \right)u_{\varepsilon x_i}H\left(u_{\varepsilon  x_i}\right)\right|\text{d}x\notag\\
                  \leq & {4\lambda_1}\int_{B_{rs}(x_0)}   \left|u_{\varepsilon x_i} \right|\cdot\left( 1+3 u_{\varepsilon} ^2\right)2^{p-1}\sigma  \text{d}x\notag\\
                  \leq & {2^{p+1}\sigma \lambda _1}\int_{B_{rs}(x_0)}  |\nabla u_{\varepsilon}| \cdot\left( 1+3 u_{\varepsilon} ^2\right)\text{d}x\notag\\
                  \leq &  {C(c_1,N,p,\lambda_1)\sigma r^Ns^N}.
              \end{align}

              Thus, for sufficiently small $\varepsilon $, by \eqref{lebes1}, \eqref{3122}, \eqref{3222}, {\eqref{3322}, and \eqref{3422}}, we obtain
              \begin{align*}
                \int_{0}^{1}  \int_{B_{rs}(x_0)\cap O_{\sigma } }\left(\left(\varepsilon +|\nabla u_\varepsilon| ^2\right)^\frac{p-2}{2}\left|D^2 u_{\varepsilon }\right| \right)^2\text{d}x\text{d}s
                \leq C( c_1,C_3,r_2,N,p,\lambda_1)\sigma r^{N-1},
              \end{align*}
             which, along with Lemma~\ref{erjie}, leads to
          \begin{align*}
              \int_{0}^{1}  \int_{B_{rs}(x_0)\cap O_{\sigma   }}     \frac{\left(4\lambda_1 \left(-u_\varepsilon^+ +\left(u_\varepsilon^+\right) ^3\right)+\lambda_2\vartheta _\varepsilon (u_\varepsilon )\right)^2}{(p+1)^2}\text{d}x\text{d}s\leq {C( c_1,C_3,r_2,N,p,\lambda_1)\sigma r^{N-1}}.
          \end{align*}

 {Letting $\varepsilon \rightarrow 0$,   we obtain}
             \begin{align*}
                  \int_{0}^{1}  \int_{B_{rs}(x_0)\cap O_{\sigma   }}     \frac{\left(4\lambda _1\left(-u^++(u^+)^3\right)+\lambda_2\chi _{\{u>0\}}\right)^2}{(p+1)^2}\text{d}x\text{d}s
              \leq   {C( c_1,C_3,r_2,N,p,\lambda_1)\sigma r^{N-1}}.
            \end{align*}

Note  that  {Theorem~\ref{the31} ensures that} $|u|\leq {C_0r_2^{\frac{p}{p-1}}}\leq C_0\frac{\lambda_2}{16\lambda_1 C_0}\leq\frac{\lambda_2}{16\lambda_1 } $ in $B_{r_2}$ for $\lambda_1>0$  and $\lambda_2>0$. Then, for any {$\lambda_1> 0$} and $\lambda_2>0$, we always have \begin{align*}
4\lambda _1{\left(-u+u^3\right)+\lambda_2}\geq \frac{3\lambda_2}{4}\ \text{in}\  B_{r_2}.
\end{align*}   Therefore, we obtain
          \begin{align*}
            \int_{0}^{1}  \int_{B_{rs}(x_0)\cap O_{\sigma   }\cap \{u>0 \}}   \frac{9\lambda_2^2}{16(p+1)^2}\text{d}x\text{d}s
              \leq &\int_{0}^{1}  \int_{B_{rs}(x_0)\cap O_{\sigma   }\cap \{u>0 \}}   \frac{\left(4\lambda _1{\left(-u+u^3\right)+\lambda_2}\right)^2}{(p+1)^2}\text{d}x\text{d}s
             \leq C\sigma   r^{N-1}.
             \end{align*}
{It follows that}
          \begin{align*}
              \int_{0}^{1}  \mathcal{L}^N(B_{rs}(x_0)\cap O_{\sigma   }\cap \{u> 0 \}) \text{d}s \leq \frac{16}{9}C\sigma  {(p+1)^2}\lambda_2^{-2}r^{N-1}=:{C_5 \sigma  r^{N-1}}.
          \end{align*}

         {\textbf{Case 2: $p\in\left(\frac{2N}{N+2},2\right)$}.} Let
          \begin{equation*}
           H(\eta ){:=}\left\{\begin{aligned}
           & 2\sigma ^\frac{1}{p-1}\left(\varepsilon +4\sigma  ^\frac{2}{p-1}\right)^\frac{p-2}{2},&& \eta > 2\sigma  ^\frac{1}{p-1},\\
           & (\varepsilon +\eta^2) ^\frac{p-2}{2}\eta  ,&& |\eta| \leq 2\sigma  ^\frac{1}{p-1},\\
           &-2\sigma ^\frac{1}{p-1}\left(\varepsilon +4\sigma  ^\frac{2}{p-1}\right)^\frac{p-2}{2},&& \eta <- 2\sigma  ^\frac{1}{p-1}.\\
           \end{aligned}\right.
           \end{equation*}

Multiplying  \eqref{dxi22} by $H\left(u_{\varepsilon  x_i}\right)$ and integrating over $B_{rs}(x_0)$ {with $r\in(0,r_2)$}, we have
          \begin{align}\label{lebes2}
            &  \int_{B_{rs}(x_0)}     \left((p-2)\left(\varepsilon +|\nabla u_\varepsilon| ^2\right)^\frac{p-4}{2}  \nabla u_\varepsilon\nabla u_{\varepsilon x_i}\nabla u_{\varepsilon }  +  \left(\varepsilon +|\nabla u_\varepsilon| ^2\right)^\frac{p-2}{2}  \nabla u_{\varepsilon x_i} \right)  \nabla H\left(u_{\varepsilon  x_i}\right)\text{d}x\notag\\
            &  +  {4\lambda_1}\int_{B_{rs}(x_0)}  \!  \left(- \chi_{\{u_{\varepsilon}>0\}} +3\left(u_\varepsilon^+\right)^2 \right)u_{\varepsilon x_i} H\left(u_{\varepsilon  x_i}\right)\text{d}x  +  {\lambda_2}\int_{B_{rs}(x_0)}    {\vartheta' _\varepsilon (u_\varepsilon )}u_{\varepsilon x_i}H\left(u_{\varepsilon  x_i}\right)\text{d}x\notag\\
           =  &\int_{\partial B_{rs}(x_0)}  \!   \left((p-2)\left(\varepsilon +|\nabla u_\varepsilon| ^2\right)^\frac{p-4}{2}  \nabla u_\varepsilon\nabla u_{\varepsilon x_i}\nabla u_{\varepsilon}  +  \left(\varepsilon +|\nabla u_\varepsilon| ^2\right)^\frac{p-2}{2}   \nabla u_{\varepsilon x_i}  \right)   H\left(u_{\varepsilon  x_i}\right){\boldsymbol{\nu}} \text{d}S.
          \end{align}

  {By  the H\"{o}lder's inequality and \eqref{lr224a}, we deduce that}
          \begin{align}\label{31224}
           &\int_{0}^{1}  \int_{\partial B_{rs}(x_0)}   \left((p-2)\left(\varepsilon +|\nabla u_\varepsilon| ^2\right)^\frac{p-4}{2}\nabla u_\varepsilon\nabla u_{\varepsilon x_i}\nabla u_{\varepsilon }+\left(\varepsilon +|\nabla u_\varepsilon| ^2\right)^\frac{p-2}{2} \nabla u_{\varepsilon x_i}\right) H\left(u_{\varepsilon  x_i}\right){\boldsymbol{\nu}} \text{d}S  \text{d}s \notag\\
           \leq & \int_{0}^{1}  \int_{\partial B_{rs}(x_0)}   \left|(p-2)\left(\varepsilon +|\nabla u_\varepsilon| ^2\right)^\frac{p-4}{2}\nabla u_\varepsilon\nabla u_{\varepsilon x_i}\nabla u_{\varepsilon }+\left(\varepsilon +|\nabla u_\varepsilon| ^2\right)^\frac{p-2}{2} \nabla u_{\varepsilon x_i}\right| H\left(u_{\varepsilon  x_i}\right){\boldsymbol{\nu}} \text{d}S  \text{d}s \notag\\
           \leq & (p+1)\int_{B_{r}(x_0)}  \left(\varepsilon +|\nabla u_\varepsilon| ^2\right)^\frac{p-2}{2}\left|D^2u_\varepsilon \right|\cdot |H\left(u_{\varepsilon  x_i}\right)|\text{d}x \notag\\
           \leq & (p+1)  \left(\int_{B_{r}(x_0)}  \left(\left(\varepsilon +|\nabla u_\varepsilon| ^2\right)^\frac{p-2}{2}|D^2u_\varepsilon |\right)^2 \text{d}x\right)^{\frac{1}{2}} \left(\int_{B_{r}(x_0)}  |H\left(u_{\varepsilon  x_i}\right)|^2 \text{d}x\right)^{\frac{1}{2}}\notag\\
           \leq & (p+1) C_4^{\frac{1}{2}} r^{\frac{N-2}{2}}2\sigma ^\frac{1}{p-1}\left(\varepsilon +4\sigma  ^\frac{2}{p-1}\right)^\frac{p-2}{2}C(N)r^\frac{N}{2}\notag\\
           \leq & {(p+1) C_4^{\frac{1}{2}} r^{\frac{N-2}{2}}2^{p-1}\sigma  C(N)r^\frac{N}{2}}\notag\\
          \leq & {C(C_4,N,p)\sigma  r^{N-1}}.
           \end{align}

 {We infer from $p\in \left(\frac{2N}{N+2},2\right)$ and $O_\varepsilon \subset O_{\varepsilon i}$ that}
           \begin{align}\label{32224}
               &  \sum_{i = 1}^{N}  \int_{B_{rs}(x_0)}     \left((p-2)\left(\varepsilon +|\nabla u_\varepsilon| ^2\right)^\frac{p-4}{2}  \nabla u_{\varepsilon }\nabla u_{\varepsilon x_i}\nabla u_{\varepsilon}+\left(\varepsilon +| \nabla u_\varepsilon| ^2\right)^\frac{p-2}{2}  \right)\nabla H\left(u_{\varepsilon  x_i}\right)\text{d}x\notag\\
               \geq& \sum_{i = 1}^{N}  \int_{B_{rs}(x_0)\cap O_{\varepsilon i}}\left((p-2)\left(\varepsilon +|\nabla u_\varepsilon| ^2\right)^\frac{p-4}{2}\left|\nabla u_{\varepsilon }\nabla u_{\varepsilon x_i}\right|^2\!\!+\left(\varepsilon +|\nabla u_\varepsilon| ^2\right)^\frac{p-2}{2} \left|\nabla u_{\varepsilon x_i}\right|^2\right)\notag\\
               &\qquad \qquad\qquad\times  \left((p-2)\left(\varepsilon +u_{\varepsilon x_i}^2\right)^\frac{p-4}{2}u_{\varepsilon x_i}^2+\left(\varepsilon +u_{\varepsilon x_i}^2\right)^\frac{p-2}{2}\right)\text{d}x\notag\\
               \geq &(p-1)^2\sum_{i = 1}^{N}  \int_{B_{rs}(x_0)\cap O_{\varepsilon i}}   \left(\varepsilon +|\nabla u_\varepsilon| ^2\right)^\frac{p-2}{2}\left|\nabla u_{\varepsilon x_i}\right|^2\left(\varepsilon +u_{\varepsilon x_i}^2\right)^\frac{p-2}{2}\text{d}x\notag\\
               \geq &(p-1)^2\sum_{i = 1}^{N}  \int_{B_{rs}(x_0)\cap O_{\varepsilon }}   \left(\varepsilon +|\nabla u_\varepsilon| ^2\right)^{p-2}\left|\nabla u_{\varepsilon x_i}\right|^2\text{d}x\notag\\
               = &(p-1)^2\int_{B_{rs}(x_0)\cap O_{\varepsilon }}   \left(\varepsilon +|\nabla u_\varepsilon| ^2\right)^{p-2}\left|D^2 u_{\varepsilon }\right|^2 \text{d}x\notag\\
               \geq &(p-1)^2\int_{B_{rs}(x_0)\cap O_{\sigma   }}  \left(\varepsilon +|\nabla u_\varepsilon| ^2\right)^{p-2}\left|D^2 u_{\varepsilon }\right| ^2\text{d}x.
           \end{align}

Since $\vartheta _\varepsilon (t)$ is non-decreasing in $t$ and $H(\eta )\eta \geq 0$, we get
           \begin{align}\label{33224}
               {\lambda_2}\int_{B_{rs}(x_0)} \vartheta' _\varepsilon (u_\varepsilon )u_{\varepsilon x_i}H\left(u_{\varepsilon  x_i}\right)\text{d}x\geq 0.
           \end{align}

 By the definition of $H$ and $\left(\|   u_\varepsilon\|_{L^\infty(B_{r_2})}+\| \nabla u_\varepsilon\|_{L^\infty(B_{r_2})} \right)\leq c_1$, it follows that
           \begin{align}\label{34224}
                \left|{4\lambda_1}\int_{B_{rs}(x_0)}  \left(- \chi_{\{u_{\varepsilon}>0\}} +3\left(u_\varepsilon^+\right)^2 \right)u_{\varepsilon x_i} H\left(u_{\varepsilon  x_i}\right)\text{d}x\right|
               \leq &{4\lambda_1}\int_{B_{rs}(x_0)}  \left|\left(- \chi_{\{u_{\varepsilon}>0\}} +3\left(u_\varepsilon^+\right)^2 \right)u_{\varepsilon x_i} H\left(u_{\varepsilon  x_i}\right)\right|\text{d}x\notag\\
               \leq & {4\lambda_1}\int_{B_{rs}(x_0)}\left|u_{\varepsilon x_i} \right|\cdot\left| 1+3 u_{\varepsilon}   ^2\right|\cdot2\sigma ^\frac{1}{p-1}\left(\varepsilon +4\sigma  ^\frac{2}{p-1}\right)^\frac{p-2}{2} \text{d}x\notag\\
               \leq & {2^{p+1}\sigma \lambda_1}\int_{B_{rs}(x_0)}   |\nabla u_{\varepsilon}|\cdot\left|1+3 u_{\varepsilon}   ^2\right| \text{d}x\notag\\
               \leq &  {C(c_1,N,p,\lambda_1)\sigma r^Ns^N}.
           \end{align}

Thus, for {sufficiently small $\varepsilon $}, we deduce by \eqref{lebes2}, \eqref{31224}, \eqref{32224}, {\eqref{33224}, and} \eqref{34224} that
           \begin{align*}
              \int_{0}^{1}  \int_{B_{rs}(x_0)\cap O_{\sigma   } }   \left(\left(\varepsilon +|\nabla u_\varepsilon|^2 \right)^\frac{p-2}{2}\left|D^2 u_{\varepsilon }\right| \right)^2\text{d}x\text{d}s
                \leq {C(c_1,C_4,r_2,N,p,\lambda_1)}\sigma r^{N-1} .
              \end{align*}

         {Then, proceeding in the same way as in Case I, we obtain}
         \begin{align*}
         \int_{0}^{1}  \mathcal{L}^N(B_{rs}(x_0)\cap O_{\sigma   }\cap \{u> 0 \}) \text{d}s \leq \frac{16}{9}C\sigma  {(p+1)^2}\lambda_2^{-2}r^{N-1}=:{C_6  \sigma  r^{N-1}}.
         \end{align*}
            \end{proof}

            \begin{theorem}\label{t41} Assume that     $\lambda_2>0$.
             Let $r_2$ be a positive constant {determined by} Lemma~\ref{LL22}.
            {Then, there} exists at least one minimizer of $\mathcal{J} (u)$ over the set $\mathcal{K} $ such that, for any $x_0\in \partial  \{u>0\}\cap {B_{r_2}}$ and {$r\in{(0,\frac{r_2}{2})}$} there holds
         \begin{align*}
             \mathcal{H} ^{N-1}( B_r(x_0)\cap\Gamma ^+)\leq {Cr^{N-1}},
         \end{align*}
   where $C$ is a {positive} constant depending only on $r_2$, $N$, $p$, $\lambda_1$, $\lambda_2$, $\| g\| _{{L^\infty}(\Omega )}$, {$\| g\| _{W^{1.p}(\Omega )}$,
             and} {the} diameter of $\Omega $.
            \end{theorem}
            \begin{proof}
             Under the conditions of Lemma~\ref{LL22}, {we claim that}
             \begin{align}\label{hausdorff1}
                 \mathcal{L}^N\left(O_\sigma  \cap B_{r}(x_0)\cap \{u>0\}\right)\leq {C_7 \sigma  r^{N-1}},\ \forall {r<\frac{r_2}{2}},
             \end{align}
             where $C_7$ is a {positive} constant depending only on $r_2$, $N$, $p$, $\lambda_1$, $\lambda_2$, $\| g\| _{{L^\infty}(\Omega )}$, {$\| g\| _{W^{1.p}(\Omega )}$,
             and} {the} diameter of $\Omega $.

           We prove by contradiction. Let us suppose {\eqref{hausdorff1} fails}. Then, there exists a ball $B_r(x_0)$ with center on the free boundary such that for any $k\in \mathbb{R}$, it holds that
            \begin{equation*}
             \mathcal{L}^N\left(O_\sigma  \cap B_{r}(x_0)\cap \{u>0\}\right)\geq {k \sigma  r^{N-1}}.
            \end{equation*}
             However, by Lemma~\ref{LL22}, we have
             \begin{align*}
                  C{\sigma r^{N-1}}
                 \geq &\int_{0}^{1}  \mathcal{L}^N\left(O_\sigma  \cap B_{2rs}(x_0)\cap \{u>0\}\right)\text{d}s\\
                  =&\int_{0}^{\frac{1}{2}}    \mathcal{L}^N\left(O_\sigma  \cap B_{2rs}(x_0)\cap \{u>0\}\right)\text{d}s+\int_{\frac{1}{2}}^{1}    \mathcal{L}^N\left(O_\sigma  \cap B_{2rs}(x_0)\cap \{u>0\}\right)\text{d}s\\
                  \geq &\frac{1}{2}\mathcal{L}^N\left(O_\sigma  \cap B_{r}(x_0)\cap \{u>0\}\right)\\
                  \geq & {\frac{1}{2}k\sigma r^{N-1}},
             \end{align*}
              which is a contradiction for $k\rightarrow \infty$.

              {Now, according to Proposition~\ref{p33}}, {for} $x_0\in B_{1-\sigma  }\cap\Gamma ^+$, there exist $y_0\in \{u>0\}$ and $c(N,p)>0$ such that
                 $B_{c\sigma  }(y_0)\subset\left( B_\sigma  (x_0)\cap O_\sigma  \cap  \{u>0\}\right)$.
{By virtue of the Besicovitch covering theorem, letting $\left\{B_\sigma  (x^i)\right\}_{i\in I}$ be finite coverings of $B_r(x_0)\cap\Gamma ^+$ with $x^i\in \Gamma ^+$ and}
                 at most $n(N)$ overlapping at each point, we have
                 \begin{align*}
                     \sum_{i\in I}  C(N)(C\sigma  )^N
                     \leq& \sum_{i\in I}\mathcal{L}^N\left(O_\sigma  \cap B_{\sigma  }\left(x^i\right)\cap \{u>0\}\right)
                      \leq n\mathcal{L}^N\left(O_\sigma  \cap B_{r }(x_0)\cap \{u>0\}\right)
                      \leq {C_7 n \sigma r^{N-1}}.
                 \end{align*}
             Therefore, {there holds}
                 $\mathcal{H} ^{N-1}(B_r(x_0)\cap\Gamma ^+)\leq \displaystyle\liminf _{\sigma  \rightarrow 0}C(N)\sigma^{N-1}\leq {C r^{N-1}}$.
             \end{proof}

\end{document}